\documentclass[10pt,a4paper]{article}
\usepackage[a4paper,top=3cm,bottom=3cm,left=3cm,right=3cm,%
bindingoffset=5mm]{geometry}
\pdfoutput=1
\usepackage[utf8]{inputenc}
\usepackage[english]{babel}
\usepackage{amsmath}
\usepackage{amsfonts}
\usepackage{amssymb}
\usepackage{graphicx}
\usepackage{authblk}
\usepackage{amsthm}
\usepackage{epstopdf}
\usepackage{tabularx}
\usepackage{multirow}
\usepackage{booktabs}
\usepackage{caption}
\usepackage{subcaption}
\usepackage[percent]{overpic}
\newcommand{\numberset}{\mathbb}

\newtheorem{theorem}{Theorem}[section]

\newtheorem{lemma}[theorem]{Lemma}
\newtheorem{remark}[theorem]{Remark}
\newtheorem{proposition}{Proposition}[section]

\newcommand{\cfan}[1]{\textnormal{\textbf{\detokenize{#1}}}}

\newcommand{\N}{\numberset{N}}
\newcommand{\R}{\numberset{R}}

\newcommand{\Pk}{\numberset{P}}
\newcommand{\Pkc}{\numberset{P}^{\rm cont}}

\newcommand{\diver}{{\rm div}}
\newcommand{\bdiver}{\boldsymbol{\diver}}
\newcommand{\bnabla}{\boldsymbol{\nabla}}

\newcommand{\bcurl}{\boldsymbol{{\rm curl}}}
\newcommand{\beps}{\boldsymbol{\epsilon}}

\newcommand{\jump}[1]{\lbrack\!\lbrack\,#1\,\rbrack\!\rbrack}
\newcommand{\media}[1]{\left\{\!\left\{\,#1\,\right\}\!\right\}}
\newcommand{\jumpmedia}[1]{\left(\!\left(\,#1\,\right)\!\right)}

\newcommand{\normas}[1]{\left\|#1\right\|_{\rm S}}
\newcommand{\normam}[1]{\left\|#1\right\|_{\rm M}}
\newcommand{\normaupw}[1]{\left|#1\right|_{\rm upw}}
\newcommand{\normacip}[1]{\left|#1\right|_{\rm cip}}
\newcommand{\normastab}[1]{\left\|#1\right\|_{\rm stab}}
\newcommand{\inpr}[1]{\left(\!\left(\,#1\,\right)\!\right)}
\newcommand{\normai}[1]{{\left\vert\kern-0.25ex\left\vert\kern-0.25ex\left\vert#1\right\vert\kern-0.25ex\right\vert\kern-0.25ex\right\vert}}

\newcommand{\nn}{\boldsymbol{n}}

\newcommand{\Edges}{\Sigma_h}
\newcommand{\EdgesB}{\Sigma_h^{\partial}}
\newcommand{\EdgesI}{\Sigma_h^{\rm {int}}}
\newcommand{\EdgesE}{\Sigma_h^E}
\newcommand{\Edgesloc}{\Sigma_h^{\omega_E}} 

\renewcommand{\epsilon}{\varepsilon}
\renewcommand{\theta}{\vartheta}
\renewcommand{\rho}{\varrho}
\renewcommand{\phi}{\varphi}
\newcommand{\zz}{\boldsymbol{\zeta}}

\newcommand{\uuht}{\widetilde{\boldsymbol{u}}_h}

\newcommand{\BB}{\boldsymbol{B}}
\newcommand{\HH}{\boldsymbol{H}}
\newcommand{\GG}{\boldsymbol{G}}
\newcommand{\TT}{\boldsymbol{\Theta}}

\newcommand{\uu}{\boldsymbol{u}}
\newcommand{\vv}{\boldsymbol{v}}
\newcommand{\ww}{\boldsymbol{w}}
\newcommand{\cc}{\boldsymbol{\chi}}
\newcommand{\ff}{\boldsymbol{f}}

\newcommand{\qq}{\boldsymbol{q}}
\newcommand{\pp}{\boldsymbol{p}}
\newcommand{\xx}{\boldsymbol{x}}

\newcommand{\pss}{\sigma_{\rm S}}
\newcommand{\psm}{\sigma_{\rm M}}
\newcommand{\ns}{\nu_{\rm S}}
\newcommand{\nm}{\nu_{\rm M}}
\newcommand{\bdma}{\mu_a}
\newcommand{\bdmc}{\mu_c}
\newcommand{\bdmJ}{\mu_{J_1}}
\newcommand{\bdmJJ}{\mu_{J_2}}
\newcommand{\Huno}{\boldsymbol{H}^1(\Omega)}
\newcommand{\Hunoenne}{\boldsymbol{H}^1_{\nn}(\Omega)}
\newcommand{\WWc}{\boldsymbol{W}}
\newcommand{\WWd}{\boldsymbol{W}^h_k}
\newcommand{\Hunozero}{\boldsymbol{H}^1_0(\Omega)}
\newcommand{\VVc}{\boldsymbol{V}}
\newcommand{\VVs}{\boldsymbol{V_*}}
\newcommand{\VVd}{\boldsymbol{V}^h_k}
\newcommand{\ZZc}{\boldsymbol{Z}}
\newcommand{\ZZs}{\boldsymbol{Z_*}}
\newcommand{\ZZd}{\boldsymbol{Z}^h_k}
\newcommand{\Qc}{Q}
\newcommand{\Qd}{Q^h_{k}}

\newcommand{\mesh}[1]{\texttt{mesh #1}}
\newcommand{\mfs}{\texttt{mfStab}}
\newcommand{\fs}{\texttt{fStab}}

\newcommand{\am}{a^{\rm M}}
\newcommand{\as}{a^{\rm S}}

\newcommand{\ash}{a^{\rm S}_h}
\newcommand{\astab}{\mathcal{A}_{\rm stab}}

\newcommand{\ccoe}{c_{\rm coe}}

\newcommand{\intcip}{\mathcal{I}_{\mathcal{O}}}
\newcommand{\Pzerok}[1]{\Pi_{#1}}
\newcommand{\PWWd}{\mathcal{I}_{\WWc}}
\newcommand{\PVVd}{\mathcal{I}_{\VVc}}
\newcommand{\omegazh}{\Omega_h^{\boldsymbol{\zeta}}}
\newcommand{\omegaEh}{\Omega_h^{M}}
\newcommand{\omegaz}{\omega_{\boldsymbol{\zeta}}}

\newcommand{\uui}{\PVVd \uu}
\newcommand{\vvi}{\PVVd \vv}
\newcommand{\eei}{\boldsymbol{e}_{\mathcal{I}}}
\newcommand{\eeh}{\boldsymbol{e}_h}
\newcommand{\BBi}{\PWWd \BB}
\newcommand{\HHi}{\PWWd \HH}
\newcommand{\EEi}{\boldsymbol{E}_{\mathcal{I}}}
\newcommand{\EEh}{\boldsymbol{E}_h}

\newcommand{\lm}{\Lambda_{\rm M}}
\newcommand{\ls}{\Lambda_{\rm S}}
\newcommand{\gm}{\Gamma_{\rm M}}
\newcommand{\gs}{\Gamma_{\rm S}}
\newcommand{\ts}{\Phi_{\rm S}}
\newcommand{\ps}{\Psi_{\rm S}}

\usepackage{color}

\title{Robust Finite Elements for linearized Magnetohydrodynamics}

\author[1,2]{L. Beir\~ao da Veiga \thanks{lourenco.beirao@unimib.it}}

\author[1]{F. Dassi \thanks{franco.dassi@unimib.it}}

\author[3]{G. Vacca \thanks{giuseppe.vacca@uniba.it}}

\affil[1]{Dipartimento di Matematica e Applicazioni,
Universit\`a degli Studi di Milano Bicocca,
Via Roberto Cozzi 55 - 20125 Milano, Italy}

\affil[2]{IMATI-CNR, Via Ferrata 2, 20127 Pavia}

\affil[3]{Dipartimento di Matematica, 
Universit\`a degli Studi di Bari, 
Via Edoardo Orabona 4  - 70125 Bari, Italy}

\begin{document}
\maketitle

\abstract{
We introduce a pressure robust Finite Element Method for the linearized Magnetohydrodynamics equations in three space dimensions, which is provably quasi-robust also in the presence of high fluid and magnetic Reynolds numbers. The proposed scheme uses a non-conforming BDM approach with suitable DG terms for the fluid part, combined with an $H^1$-conforming choice for the magnetic fluxes. The method introduces also a specific CIP-type stabilization associated to the coupling terms. Finally, the theoretical result are further validated by numerical experiments.
}

\section{Introduction}\label{sec:intro}

The research area of magnetohydrodynamics (MHD) has attracted increasing attention in the community of computational mathematics in these recent years. Such kind of equations arise, for instance, in the study of plasmas and liquid metals, and have applications in numerous fields including geophysics, astrophysics, and engineering.
The combination of equations stemming from different areas (namely fluido-dynamics and electro-magnetism) leads to a variety of models, with different formulations and different finite element choices , yielding a very ample plethora of methods with associated assets and drawbacks. Such variety includes also other critical aspects such as the choice for explicit or implicit time advancement and a wide range of linear/nonlinear solver techniques (a non-exhaustive list of contributions being \cite{Gerbeau2,Guermond,Schotzau,Greif,Prohl,Houston,Dong,Perugia1,Perugia2,Hiptmair,Wacker, Badia, Zhang,VEM}).

The initial motivation of this article is the unsteady MHD problem (see for instance \cite{Gerbeau}) in a three-field formulation, here opportunely scaled for simplicity of exposition:
\begin{equation}
\label{eq:non-linear primale}
\left\{
\begin{aligned}
&\begin{aligned}
\uu_t -  \ns \, \bdiver (\beps (\uu) ) + (\bnabla \uu ) \,\uu    +  \BB \times \bcurl(\BB)  - \nabla p &= \ff \qquad  & &\text{in $\Omega \times I$,} 
\\
\diver \, \uu &= 0 \qquad & &\text{in $\Omega \times I$,} 
\\
\BB_t +  \nm \, \bcurl (\bcurl (\BB) ) - \bcurl(\uu \times \BB)  &= \GG \qquad  & &\text{in $\Omega \times I$,} 
\\
\diver \, \BB &= 0 \qquad & &\text{in $\Omega \times I$,} 
\end{aligned}
\end{aligned}
\right.
\end{equation}
where the unknowns of the problem are $\uu$, $p$ and $\BB$, representing the velocity field, the fluid pressure and the magnetic induction, respectively; the parameters $\ns$, $\nm$ $\in \R^+$ denote the fluid and magnetic diffusive coefficients, while $\ff$, $\GG \in \boldsymbol{L}^2(\Omega)$ stand for the volumetric external forces.

In many engineering and physical applications of practical interest the (scaled) parameters $\ns$ and, possibly, $\nm$ are substantially small. For instance in aluminium electrolysis $\ns$ is about \texttt{1e-5} and $\nm$ about \texttt{1e-1} (see for instance~\cite{Berton, Armero}), while much smaller values also for $\nm$ can be reached in problems stemming from space weather prediction.

It is well known that finite element schemes in fluidodynamics may suffer from instabilities when the convective term is dominant with respect to the diffusive term. In 
such situations a stabilization is required in order to obtain reliable numerical results (among the wide literature we refer to the book \cite{volker:book}, the review \cite{garcia} or the few sample papers \cite{BH:1982,BFH:2006,OLHL:2009,MT:2015,BDV:NS-TD}). 

In the more complex case of the MHD equations, this phenomenon appears also with respect to the electro-magnetic part of the problem. If no specific care is taken in this respect, even for moderately small values of the diffusion parameters the accuracy of the velocity solution can strongly suffer. Another aspect which is recognized as critical in the modern literature of incompressible fluids discretization is that of pressure robustness, see for instance~\cite{linke-merdon:2016,john-linke-merdon-neilan-rebholz:2017}.

In this contribution we focus on a linearized version of \eqref{eq:non-linear primale} and develop an arbitrary order numerical scheme 
with the following characteristics: (1) pressure robust; (2) quasi-robust with respect to dominant advection, by which we mean that assuming a sufficiently regular solution we obtain estimates that are independent of both $\ns,\nm$ in an error norm that includes control on convection; (3) differently from SUPG or similar approaches, the method is suitable for extension to the time-dependent case via a standard time-stepping scheme.
To achieve these goals we start by using a divergence-free non-conforming BDM element for the velocity/pressure couple, combined with an upwinding DG technique for consistency and robustness. Such element is considered among the best for incompressible fluid mechanics (see for instance \cite{Hdiv1,Hdiv2,Hdiv3}). In order to keep the magnetic part simple and efficient, we assume a convex domain and an $H^1$-conforming discretization of the magnetic fluxes. Furthermore, to handle the fluido-magnetic coupling terms in a robust way in all regimes, we introduce a specific novel CIP-type stabilization. In the ensuing theoretical analysis we make use of a specific interpolant which satisfies suitable {\it local} inf-sup and approximation properties, thus allowing us to avoid a more involved Nitsche type imposition of the boundary conditions and, most importantly, a quasi-uniformity assumption on the mesh. Both our theory and numerical results are developed in three space dimensions.

The paper is organized as follows. In Section \ref{sec:cont} we introduce the continuous problem. 
In Section \ref{sec:notations}, after describing some notation, we introduce some preliminary result.
In Section \ref{sec:stab} we present our proposed stabilized scheme. In Section \ref{sec:theo} we prove the theoretical convergence results. Finally, numerical tests are shown in Section \ref{sec:num}.

%
%
%
%
%
%
%
%
%
%
%
%
%
%
%
%
%
%

\section{Continuous problem}\label{sec:cont}

We start this section with some standard notations. 
Let the computational domain $\Omega \subset \R^3$ be a convex polyhedron with regular boundary $\partial \Omega$ having outward pointing unit normal $\nn$.
%
The symbols $\nabla$ and $\Delta$    
denote the gradient and the Laplacian operator for scalar functions, respectively, while  
$\bnabla$, $\beps$,  $\bcurl$ and $\diver$ 
denote   
the gradient, the symmetric gradient operator, the curl operator,
and  the divergence operator for vector valued functions. Finally, 
$\bdiver$ denotes the vector valued divergence operator for tensor fields. 

Throughout the paper, we will follow the usual notation for Sobolev spaces
and norms \cite{Adams:1975}.
Hence, for an open bounded domain $\omega$,
the norms in the spaces $W^r_p(\omega)$ and $L^p(\omega)$ are denoted by
$\|{\cdot}\|_{W^r_p(\omega)}$ and $\|{\cdot}\|_{L^p(\omega)}$, respectively.
Norm and seminorm in $H^{r}(\omega)$ are denoted respectively by
$\|{\cdot}\|_{r,\omega}$ and $|{\cdot}|_{r,\omega}$,
while $(\cdot,\cdot)_{\omega}$ and $\|\cdot\|_{\omega}$ denote the $L^2$-inner product and the $L^2$-norm (the subscript $\omega$ may be omitted when $\omega$ is the whole computational
domain $\Omega$).
For the functional spaces introduced above we use the bold symbols to denote the corresponding sets of vector valued functions.
Finally, we introduce the following spaces
\[
\begin{aligned}
\Hunoenne &:= \{\vv \in \Huno \,\,\, \text{s.t.} \,\,\, \vv \cdot \nn = 0 \,\,\, \text{on $\partial \Omega$} \} \,, \\
\HH_0(\diver, \Omega) &:= \{\vv \in \boldsymbol{L}^2(\Omega) \,\,\,\text{s.t.} \,\,\, \diver \,\vv \in L^2(\Omega) \,\,\, \text{and} \,\,\, \vv \cdot \nn = 0 \,\,\, \text{on $\partial \Omega$} \} \,.
\end{aligned}
\]
We consider the following linearized version of Problem~\eqref{eq:non-linear primale}:
\begin{equation}
\label{eq:linear primale}
\left\{
\begin{aligned}
& \text{ find $(\uu, p, \BB)$ such that}\\
&\begin{aligned}
\pss \, \uu -  \ns \, \bdiver (\beps (\uu) ) + (\bnabla \uu ) \,\cc    +  \TT \times \bcurl(\BB) - \nabla p &= \ff \qquad  & &\text{in $\Omega$,} 
\\
\diver \, \uu &= 0 \qquad & &\text{in $\Omega$,} 
\\
\psm \, \BB +  \nm \, \bcurl (\bcurl (\BB) ) - \bcurl(\uu \times \TT)  &= \GG \qquad  & &\text{in $\Omega$,} 
\\
\diver \, \BB &= 0 \qquad & &\text{in $\Omega$,} 
\end{aligned}
\end{aligned}
\right.
\end{equation}
coupled with the homogeneous boundary conditions
\begin{equation}
\label{eq:linear bc cond}
\uu = 0 \,, \quad 
\BB \cdot \nn = 0 \,,  \quad 
\bcurl(\BB) \times \nn = 0  \quad \text{on $\partial \Omega$,}
\end{equation}
where the new parameters $\pss$, $\psm \in \R^+$ represent the reaction coefficients;
$\cc,\TT$ in $\HH_0(\diver, \Omega) \cap \boldsymbol{L}^3(\Omega)$, with $\diver \, \cc =0$ in $\Omega$, represent respectively the fluid advective field and the magnetic advective field.

Notice that the third and fourth equations in \eqref{eq:linear primale} and the boundary conditions \eqref{eq:linear bc cond} yield the compatibility condition  
$(\GG \,, \nabla \psi)  = 0$ for all $\psi \in H^1(\Omega)$.

The proposed linear problem has a practical interest not only as a simplified model but also because it can be directly derived when a time integrator scheme is applied to the unsteady nonlinear problem \eqref{eq:non-linear primale}. In that case 
$\cc$ and $\TT$ correspond to $\uu^n$ and $\BB^n$ respectively (that are assumed to be known), and we compute the approximate solution at time step $n+1$.  Therefore, by avoiding the requirement $\diver\,\TT = 0$ we are implicitly allowing for a discretization that does not enforce the magnetic divergence constraint exactly. 

We now derive the variational formulation for Problem \eqref{eq:linear primale}.
Consider the following spaces
\begin{equation}
\label{eq:spazi_c}
\VVc := \Hunozero \,, 
\quad
\WWc := \Hunoenne \,,
\quad
\Qc  := L^2_0(\Omega) = \left\{ q \in L^2(\Omega) \quad \text{s.t.} \quad (q, \,1) = 0 \right\}  \,,
\end{equation}
representing the velocity field space, the magnetic induction space and the pressure space, respectively, endowed with the standard norms, and the forms
\begin{equation}
\label{eq:forme_c1}
\as(\uu,  \vv) :=  (\beps(\uu), \, \beps (\vv) ) \,,
\qquad
c(\uu, \vv) :=  \left( ( \bnabla \uu ) \, \cc ,\, \vv  \right) \,,
\end{equation}
and 
\begin{equation}
\label{eq:forme_c2}
\begin{aligned}
\am(\BB,  \HH) &:=   \left( \bcurl(\BB) ,\, \bcurl(\HH)\right) +
\left(\diver(\BB) \,, \diver(\HH) \right),
\\
b(\vv, q) &:=  (\diver \vv, \, q) \,,
\\
d(\HH, \vv) &:=  ( \bcurl (\HH ) \times \TT ,\, \vv) \,.
\end{aligned}
\end{equation}
%
Let us introduce the kernel of the  bilinear form $b(\cdot,\cdot)$
that corresponds to the functions in $\VVc$ with vanishing divergence
\begin{equation}
\label{eq:Z}
\ZZc := \{ \vv \in \VVc \quad \text{s.t.} \quad \diver \, \vv = 0  \}\,.
\end{equation}
We consider the following variational problem
\begin{equation}
\label{eq:linear variazionale}
\left\{
\begin{aligned}
& \text{find $(\uu, p, \BB) \in \VVc \times \Qc \times \WWc$,  such that} 
\\
&\begin{aligned}
\pss (\uu, \vv) + \ns \as(\uu, \vv) + c(\uu, \vv)  
-d(\BB, \vv)  + b(\vv, p) 
&= 
(\ff, \vv)  
&\,\,\, & \text{for all $\vv \in \VVc$,} 
\\
b(\uu, q) &= 0 
&\,\,\, & \text{for all $q \in \Qc$,}
\\
\psm (\BB, \HH) + \nm \am(\BB, \HH) + d(\HH, \uu) 
&= 
(\GG, \HH)  
&\,\,\, & \text{for all $\HH \in \WWc$.}
\end{aligned}
\end{aligned}
\right.
\end{equation}

\begin{proposition}
Assume that the domain $\Omega$ is a convex polyhedron. Then
Problem \eqref{eq:linear variazionale} is well-posed. 
Additionally,  Problem \eqref{eq:linear variazionale} is a variational formulation of Problem \eqref{eq:linear primale}.
\end{proposition}

\begin{proof}
We simply sketch the proof since it is derived using similar arguments to that in \cite[Proposition 3.18, Lemma 3.19]{Gerbeau}.

\noindent 
We first address the well-posedness of \eqref{eq:linear variazionale}.
Consider the form
\[
\begin{aligned}
\mathcal{S}((\uu, \BB) \,, (\vv, \HH)) :=
\pss (\uu, \vv) &+ \ns \as(\uu, \vv) + \psm (\BB, \HH) + 
\\
& + \nm \am(\BB, \HH)  + c(\uu, \vv)  -d(\BB, \vv) + d(\HH, \uu)
\end{aligned}
\]
for all $\uu$, $\vv \in \ZZc$ and $\BB$, $\HH \in \WWd$.
Employing the classical inf–sup condition for the Stokes equations \cite{boffi-brezzi-fortin:book}, the well-posedness of Problem \eqref{eq:linear variazionale} follows from the coercivity and continuity properties of the form $\mathcal{S}((\cdot, \cdot), (\cdot, \cdot))$.
Direct computations yield
\begin{equation}
\label{eq:well2}
\mathcal{S}((\vv, \HH) \,, (\vv, \HH)) =
\pss \Vert \vv \Vert^2 + \ns \Vert \beps (\vv) \Vert^2 + 
\psm \Vert\HH\Vert^2 +  \nm \Vert \bcurl(\HH)\Vert^2 
+  \nm \Vert \diver \HH\Vert^2 \,.
\end{equation} 
We now recall that if the domain $\Omega$ is a convex polyhedron the following embedding holds \cite[Theorem 3.9]{Girault-book} and \cite{Amrouche}
\begin{equation}
\label{eq:well3}
\Vert \bcurl(\HH)\Vert^2 +  \Vert \diver \HH\Vert^2 
\geq c_\Omega \Vert \HH \Vert_{\WWc}^2 \qquad \text{for all $\HH \in \WWc$}
\end{equation} 
where here and in the following 
$c_\Omega$ denotes a generic positive constant depending only on the domain $\Omega$, that may change at each occurrence.
Therefore \eqref{eq:well2}, bound \eqref{eq:well3}, the Poincar\'e inequality and the Korn inequality imply
the coercivity property
\[
\mathcal{S}((\vv, \HH) \,, (\vv, \HH)) 
\geq 
c_\Omega \min\{\pss \,, \ns \,,  \psm \,, \nm \}
(\Vert \vv\Vert_{\VVc}^2  + \Vert \HH\Vert_{\WWc}^2 ) \,.
\]
We omit the proof of the continuity of $\mathcal{S}(\cdot, \cdot)$  with respect to the same norm, which is easy to show.
We conclude that Problem \eqref{eq:linear variazionale} is well-posed.  

We now prove that \eqref{eq:linear variazionale} is a variational formulation for Problem \eqref{eq:linear primale}. It is straightforward to see that the first and the second equation in \eqref{eq:linear variazionale} are the weak form of the Oseen type equation associated with the strong formulation in \eqref{eq:linear primale} (first and second equation) coupled with the boundary conditions on $\partial \Omega$ \eqref{eq:linear bc cond} (first equation).
In order to derive from \eqref{eq:linear variazionale}  the magnetic divergence constraint $\diver \BB = 0$, let $\phi$ be the solution of the auxiliary problem
\begin{equation}
\label{eq:well5}
 - \nm \Delta \phi + \psm \phi =  - \diver \, \BB  \quad  \text{in $\Omega$,}
\qquad
\nabla \phi \cdot \nn  = 0  \quad  \text{on $\partial \Omega$.}
\end{equation}
Notice that since $\Omega$ is a convex polyhedron, $\phi \in H^2(\Omega)$, then $\nabla \phi \in \WWc$.
We set $\HH := \nabla \phi$ in \eqref{eq:linear variazionale}. Being $\bcurl(\nabla \phi) = \boldsymbol{0}$ and recalling the compatibility condition  
$(\GG \,, \nabla \psi)  = 0$ for all $\psi \in H^1(\Omega)$, we infer
$
\psm (\BB, \nabla \phi) + \nm (\diver \BB \,, \diver (\nabla \phi)) = 
0 \,.
$
Integrating by parts and using \eqref{eq:well5} we obtain
 $\diver \BB = 0$.
The third equation in \eqref{eq:linear variazionale} is then equivalent to
\begin{equation}
\label{eq:well6}
\psm (\BB, \HH) + \nm (\bcurl(\BB), \bcurl(\HH)) + d(\HH, \uu) = 
(\GG, \HH)  
\qquad \text{for all $\HH \in \WWc$,}
\end{equation}
that yields
the weak formulation of the third equation in \eqref{eq:linear primale}.
The boundary conditions easily follow  integrating
\eqref{eq:well6} by parts.
\end{proof}

It is worth mentioning that the formulation \eqref{eq:linear variazionale} allows us to recover the divergence-free constraint for the magnetic induction $\BB$ directly in the variational problem without introducing a Lagrangian multiplier.

\begin{remark}A different variational formulation should be adopted when $\Omega$ is a non convex polyhedron. 
The above formulation is well-posed also if $\Omega$ is non convex, but 
the solution of Problem \eqref{eq:linear variazionale} may not be the physical solution of the real problem that presents, in principle,
singularities due to the re-entrant corners.
\end{remark}


\section{Notations and preliminary theoretical results}
\label{sec:notations}

We now introduce some basic tools and notations that will be useful in the construction and the theoretical analysis of the proposed stabilized method.


Let $\{\Omega_h\}_h$ be a family of  conforming decompositions of $\Omega$ into tetrahedral elements $E$ of diameter $h_E$. We denote by 
$h := \sup_{E \in \Omega_h} h_{E}$ the mesh size associated with $\Omega_h$.

\noindent
Let $\mathcal{N}_h$ be the set of internal vertices of the mesh $\Omega_h$, and for any $\zz \in \mathcal{N}_h$ we set
\[
\omegazh:= \{E \in \Omega_h \, \, \, \text{s.t.} \,\, \, \zz \in E\}\,,
\quad
\omegaz := \cup_{E \in \omegazh} E \,,
\quad
h_{\zz} := \text{diameter of $\omegaz$}
 \,.
\]

\noindent
We denote by $\Edges$ the set of faces of $\Omega_h$ divided into 
internal $\EdgesI$ and external $\EdgesB$ faces; 
for any $E \in \Omega_h$ we denote by $\EdgesE$ the set of the faces of $E$. Furthermore for any $f \in \Edges$ we denote with $h_f$ the diameter of $f$.

We make the following mesh assumptions. Note that the second condition \textbf{(MA2)}
is required \emph{only} for the analysis of the lowest order case (that is order $1$).

\smallskip
\noindent
\textbf{(MA1) Shape regularity assumption:}

\noindent
The mesh family $\left\{ \Omega_h \right\}_h$ is shape regular: it exists a  positive  constant $c_{\rm{M}}$ such that each element $E \in \{ \Omega_h \}_h$ is star shaped with respect to a ball of radius $\rho_E$ with  $h_E \leq c_{\rm{M}} \rho_E$. 

\smallskip
\noindent
\textbf{(MA2) Mesh agglomeration with stars macroelements:}

\noindent 
There exists a family of conforming meshes $\{ \widetilde{\Omega}_h \}_h$ of $\Omega$ with the following properties:
(i) it exists a positive constant $\widetilde{c}_{\rm{M}}$ such that each element $M \in \widetilde{\Omega}_h$ is a finite  (connected) agglomeration of elements in $\Omega_h$, i.e., it exists $\omegaEh \subset \Omega_h$ with ${\rm{card}}(\omegaEh) \leq \widetilde{c}_{\rm{M}}$ and $M = \cup_{E \in \omegaEh} E$;
(ii) for any $M \in \widetilde{\Omega}_h$ it exists $\zz \in \mathcal{N}_h$ such that $\omegaz \subseteq M$.

\begin{remark}
\label{rm:mesh}
Assumption \cfan{(MA1)} is classical in FEM and is needed to derive optimal approximation properties for the polynomial spaces. 
Assumption \cfan{(MA2)} is needed only for $k=1$ and has a purely theoretical purpose (see Lemma \ref{lm:int-i} and Lemma \ref{lm:cip}).  However, it is easy to see that \cfan{(MA2)} is not restrictive. Given a mesh $\Omega_h$ we simply form the elements of $\widetilde{\Omega}_h$ by taking a sub-set of the internal nodes $\zz$ of $\Omega_h$ and building the associated ``stars'' $\omegaz$ such that there is no overlap. The remaining ``gap'' elements are then attached to the already formed stars.
\end{remark}

The mesh assumption \textbf{(MA1)} easily implies  the following property.

\noindent
\textbf{(MP1) local quasi-uniformity:}

\noindent
it exists a positive constant $c_{\rm{P}}$ depending on $c_{\rm{M}}$ such that for any $E \in \Omega_h$  and $\zz \in \mathcal{N}_h$
\[
\max_{E \in \Omega_h} \max_{f \in \EdgesE} \frac{h_E}{h_f} \leq c_{\rm{P}} \,,
\quad
{\rm{card}}(\Omega_h^{\zz}) \leq c_{\rm{P}} \,,
\qquad
\max_{E', E'' \in \Omega_h^{\zz}} \frac{h_E'}{h_{E''}} \leq c_{\rm{P}} \,,
\quad
\max_{E' \in \Omega_h^{\zz}} \frac{h_{\zz}}{h_{E'}} \leq c_{\rm{P}} \,.
\]

Whereas \textbf{(MA1)} and \textbf{(MA2)} entail the following property.

\noindent
\textbf{(MP2) macroelements and stars uniformity:}

\noindent
it exists a positive constant $\widetilde{c}_{\rm{P}}$ depending on $c_{\rm{M}}$ and $\widetilde{c}_{\rm{M}}$ such that for any $M \in \widetilde{\Omega}_h$ (referring to \textbf{(MA2)}) it holds
\[
\max_{E \in \omegaEh}  \frac{h_{M}}{h_E} \leq \widetilde{c}_{\rm{P}}\,,
\qquad
\frac{h_{M}}{h_{\zz}} \leq \widetilde{c}_{\rm{P}} 
 \quad \text{if  $\omegaz \subseteq M$.} 
\]

\noindent
For $m \in \N$ and for $S \subseteq \Omega_h$, we introduce the polynomial spaces 
\begin{itemize}
\item $\Pk_m(\omega)$ is the set of polynomials on $\omega$ of degree $\leq m$, with $\omega$ a generic set;
\item $\Pk_m(S) := \{q \in L^2\bigl(\cup_{E \in S}E \bigr) \quad \text{s.t.} \quad q|_{E} \in  \Pk_m(E) \quad \text{for all $E \in S$}\}$;
\item $\Pkc_m(S) := \Pk_m(S) \cap C^0\bigl(\cup_{E \in S}E \bigr)$.
\end{itemize}
For $s \in \R^+$ and  $p \in [1,+\infty]$ let us define the broken Sobolev spaces:
\begin{itemize}
\item $W^s_p(\Omega_h) := \{\phi \in L^2(\Omega) \quad \text{s.t.} \quad \phi|_{E} \in  W^s_p(E) \quad \text{for all $E \in \Omega_h$}\}$,
\end{itemize}
equipped with  the standard broken norm 
$\Vert \cdot \Vert_{W^s_p(\Omega_h)}$ 
and seminorm 
$\vert \cdot \vert_{W^s_p(\Omega_h)}$.

\noindent
For any $E \in \Omega_h$, $\nn_E$  denotes the outward normal vector to $\partial E$.
For any mesh face $f$ let $\nn_f$ be a fixed unit normal vector to the face $f$.
Notice that for any $E \in \Omega_h$ and $f \in \EdgesE$ it holds $\nn_f = \pm \nn_E$.
We assume that for any boundary face $f \in \EdgesB \cap \EdgesE$ it holds
$\nn_f = \nn_E = \nn$, i.e. $\nn_f$ is the outward to $\partial \Omega$.

\noindent
The jump and the average operators on $f \in \EdgesI$ are defined for every piecewise continuous function w.r.t. $\Omega_h$ respectively by
\[
\begin{aligned}
\jump{\phi}_f(\xx) &:=
\lim_{s \to 0^+} \left( \phi(\xx - s \nn_f) - \phi(\xx + s \nn_f) \right)
\\
\media{\phi}_f(\xx) &:= 
\frac{1}{2}\lim_{s \to 0^+} \left( \phi(\xx - s \nn_f) + \phi(\xx + s \nn_f) \right)
\end{aligned}
\]
and $\jump{\phi}_f(\xx) = \media{\phi}_f(\xx) = \phi(\xx)$ on $f \in \EdgesB$.

\noindent
Let $\mathcal{D}$ denote one of the differential operators $\bnabla$, $\beps$, $\bcurl$.
Then, $\mathcal{D}_h$ represents the broken operator defined for all $\boldsymbol{\phi} \in \HH^1(\Omega_h)$ as
$\mathcal{D}_h (\boldsymbol{\phi} )|_E := \mathcal{D} (\boldsymbol{\phi} |_E)$ for all $E \in \Omega_h$.

\noindent
Finally, given $m \in \N$, we denote with
$\Pzerok{m} \colon L^2(\Omega_h) \to \Pk_m(\Omega_h)$ the  $L^2$-projection operator onto the space of polynomial functions.
%
The above definitions extend to vector valued and tensor valued functions.

%
%

%

In the following the symbol $\lesssim$ will denote a bound up to a generic positive constant, independent of the mesh size $h$, of the diffusive coefficients $\ns$ and $\nm$, of the reaction coefficients $\pss$ and $\psm$, of the advective fields $\cc$ and $\TT$,  of the loadings  $\ff$ and $\GG$, of the problem solution $(\uu, p, \BB)$,  but which may depend on $\Omega$, on the order
of the method $k$ (introduced in Section \ref{sec:stab}), and on the mesh regularity constants $c_{\rm{M}}$ and  $\widetilde{c}_{\rm{M}}$ in Assumptions \textbf{(MA1)} and \textbf{(MA2)}.

We close this section mentioning a list of classical results (see for instance \cite{brenner-scott:book}) that will be useful in the sequel.

\begin{lemma}[Trace inequality]
\label{lm:trace}
Under the mesh assumption \cfan{(MA1)}, for any $E \in \Omega_h$ and for any  function $v \in H^1(E)$ it holds 
\[
\sum_{f \in \EdgesE}\|v\|^2_{f} \lesssim h_E^{-1}\|v\|^2_{E} + h_E\|\nabla v\|^2_{E} \,.
\]
\end{lemma}

\begin{lemma}[Bramble-Hilbert]
\label{lm:bramble}
Under the mesh assumption \cfan{(MA1)}, let $m \in {\mathbb N}$. For any $E \in \Omega_h$ and for any  smooth enough function $\phi$ defined on $\Omega$, it holds 
\[
\Vert\phi - \Pzerok{m} \phi \Vert_{W^r_p(E)} \lesssim h_E^{s-r} \vert \phi \vert_{W^s_p(E)} 
\qquad  \text{$s,r \in \N$, $r \leq s \leq m+1$, $p \in [1, \infty]$.}
\]
\end{lemma}

\begin{lemma}[Inverse estimate]
\label{lm:inverse}
Under the mesh assumption \cfan{(MA1)}, let $m \in {\mathbb N}$. 
Then for any $E \in \Omega_h$ and for any $p_m \in \Pk_m(E)$ it holds 
\[
\|p_m\|_{W^s_p(E)} \lesssim h_E^{-s} \|p_m\|_{L^p(E)} 
\]
where the involved constant only depends on $m,s,p,c_{\rm{M}}$.
\end{lemma}

\begin{remark}
\label{rm:jumpmedia}
For any face $f \in \Edges$, let $\jumpmedia{\cdot}_f$ denote the jump or the average operator on the face $f$.
We notice that for any $\mathbb{K} \in [L^{\infty}(\Omega_h)]^{3 \times 3}$ and for any $\ww \in \boldsymbol{H}^1(\Omega_h)$ and $\alpha \in \mathbb{Z}$, mesh assumption \cfan{(MA1)} and  Lemma \ref{lm:trace} yield the following estimate
\begin{equation}
\label{eq:utile-jump1}
\begin{aligned}
\sum_{f \in \Edges} h_f^\alpha \|\jumpmedia{\mathbb{K} \ww}_f\|^2_{f}
&\lesssim  
\sum_{E \in \Omega_h} h_E^\alpha \|\mathbb{K}\|^2_{L^\infty(E)} 
\sum_{f \in \EdgesE} \|\ww\|^2_{f}
\\
& \lesssim  
\sum_{E \in \Omega_h} \|\mathbb{K}\|^2_{L^\infty(E)}  
\left( h_E^{\alpha-1} \|\ww\|^2_{E} + h_E^{\alpha+1} \|\bnabla \ww\|^2_{E} \right) \,.
\end{aligned}
\end{equation}
In particular if $\ww \in [\Pk_m(\Omega_h)]^3$ by Lemma \ref{lm:inverse} it holds that
\begin{equation}
\label{eq:utile-jump}
\sum_{f \in \Edges}  h_f^\alpha \|\jumpmedia{\mathbb{K} \ww}_f\|^2_{f}
\lesssim  
\sum_{E \in \Omega_h}  h_E^{\alpha-1} \|\mathbb{K}\|^2_{L^\infty(E)}   \|\ww\|^2_{E} \,.
\end{equation}
\end{remark}

\section{Stabilized finite element method}
\label{sec:stab}

In this section we describe the proposed stabilized method and we prove some technical results that will be useful in the convergence analysis in Section \ref{sub:err}.

\subsection{Discrete spaces and interpolation analysis}
\label{sub:spaces}

Let $k \geq 1$ be the order of the method, then we consider the following discrete spaces
\begin{equation}
\label{eq:spazi_d}
\VVd := [\Pk_k(\Omega_h)]^3 \cap \HH_0(\diver), \,\,\, 
\WWd := [\Pkc_k(\Omega_h)]^3 \cap \Hunoenne, \,\,\, 
\Qd  := \Pk_{k-1}(\Omega_h) \cap L^2_0(\Omega),
\end{equation}
approximating the velocity field space $\VVc$, the magnetic induction space $\WWc$ and the pressure space $\Qc$ respectively. 

Notice that in the proposed method we adopt the $\HH(\diver)$-conforming 
$\boldsymbol{{\rm BDM}}_k$ element \cite{BDM} for the approximation of the velocity space
that provides exact divergence-free discrete velocity,
and preserves the pressure-robustness of the resulting scheme \cite{Hdiv1,Hdiv2,Hdiv3}.
%
%
Let us introduce the discrete kernel
\begin{equation}
\label{eq:Z_h}
\ZZd := \{ \vv_h \in \VVd \quad \text{s.t.} \quad \diver \, \vv_h = 0  \}\,.
\end{equation}

We now define the interpolation operators $\PVVd$ and $\PWWd$, acting on the spaces $\VVd$ and $\WWd$ respectively, satisfying optimal approximation estimates and suitable \emph{local} orthogonality properties that will be instrumental to prove the convergence result in Section \ref{sub:err} (without the need to require a quasi-uniformity property on the meshes sequence).  
For what concerns the operator $\PVVd$, we recall from \cite{BDM} the  following result.

\begin{lemma}[Interpolation operator on $\VVd$]
\label{lm:int-v}
Under the Assumption \cfan{(MA1)} let $\PVVd \colon \VVc \to \VVd$ be the interpolation operator defined in equation (2.4) of \cite{BDM}.
The following hold

\noindent
$(i)$ if $\vv \in \ZZc$ then $\vvi \in \ZZd$;

\noindent
$(ii)$ for any $\vv \in \VVc$
\begin{equation}
\label{eq:orth-v}
\left(\vv - \, \vvi, \,  \pp_{k-1} \right) = 0 \quad \text{for all $\pp_{k-1} \in [\Pk_{k-1}(\Omega_h)]^3$;}
\end{equation}

\noindent
$(iii)$ for any $\vv \in \VVc \cap \HH^{s+1}(\Omega_h)$, with $0 \leq s \leq k$, for all $E \in \Omega_h$, it holds
\begin{equation}
\label{eq:int-v}
\vert \vv - \vvi \vert_{m,E} \lesssim h_E^{s+1-m} \vert \vv \vert_{s+1,E} 
\qquad \text{for $0\leq m\leq s+1$.}
\end{equation}
\end{lemma}

\begin{remark}
\label{rm:jumpmediaeei}
With the same notations of Remark \ref{rm:jumpmedia}, 
for any $\vv \in \VVc \cap \HH^{s+1}(\Omega_h)$, with $0 \leq s \leq k$
combining \eqref{eq:utile-jump1} and \eqref{eq:int-v}, for $\alpha = -1, 0, 1$, the following holds
\begin{equation}
\label{eq:utilejump2}
\begin{aligned}
\sum_{f \in \Edges} h_f^\alpha \|\jumpmedia{\vv - \vvi}_f\|^2_{f} +
\sum_{f \in \Edges} h_f^{\alpha+2} \|\jumpmedia{\bnabla_h(\vv - \vvi)}_f\|^2_{f}
&\lesssim   
h^{2s + 1 + \alpha} \vert \vv \vert_{s+1,\Omega_h}^2 \,. 
\end{aligned}
\end{equation}
\end{remark}

\begin{lemma}[Interpolation operator on $\WWd$]
\label{lm:int-i}
Let Assumption \cfan{(MA1)} hold. Furthermore, if $k=1$ let also Assumption \cfan{(MA2)} hold.
Then there exists an interpolation operator $\PWWd \colon \WWc \to \WWd$
satisfying the following

\noindent
$(i)$ for any $\HH \in \WWc$
\begin{equation}
\label{eq:int-orth}
\left( \HH - \HHi, \, \qq_{k-1} \right) = 0 
\qquad
\text{for any $\qq_{k-1} \in [\mathbb{O}_{k-1}(\Omega_h)]^3$}
\end{equation}
where
\begin{equation}
\label{eq:mathbbO}
\mathbb{O}_{k-1}(\Omega_h) := 
\Pkc_{k-1}(\Omega_h)  \quad  \text{for $k>1$,}
\qquad
\mathbb{O}_{k-1}(\Omega_h) := \Pk_{0}(\widetilde{\Omega}_h)  \quad \text{for $k=1$;}
\end{equation}

\noindent
$(ii)$ for any $\HH \in \WWc \cap \HH^{s+1}(\Omega_h)$ with $0 \leq s \leq k$, for $\alpha=0,1,2$, it holds

\begin{equation}
\label{eq:err-int}
\sum_{E \in \Omega_h} h_E^{-\alpha} \Vert \HH - \HHi \Vert_E^2
\lesssim h^{2s+2-\alpha} \vert \HH \vert_{s+1,\Omega_h}^2 \,,
\quad
\Vert \bnabla (\HH - \HHi) \Vert \lesssim h^{s} \vert \HH \vert_{s+1,\Omega_h} \,. 
\end{equation}

\end{lemma}

\begin{proof}
For any $\lambda \in \mathcal{N}_h \cup \widetilde{\Omega}_h$  we define the following mesh-dependent bilinear form and norms:
\[
\inpr{\uu, \, \vv}_{\lambda} := \sum_{E \in \Omega_h^\lambda} h_E^{-2}(\uu, \, \vv)_E \,,
\quad
\normai{\uu}^2_{\lambda} := \inpr{\uu, \, \uu}_{\lambda}\,,
\quad 
\normai{\uu}^2_{*,\lambda} := \sum_{E \in \Omega_h^\lambda} h_E^2 \Vert \uu \Vert_E^2 \,.
\]
%
When the subscript $\lambda$ is omitted, in the previous form and norms the subset $\Omega_h^\lambda$ is replaced by the whole mesh $\Omega_h$.
The proof follows several  steps.

\noindent 
\textsl{Step 1: Local inf-sup stability.}
Let $\zz \in \mathcal{N}_h$; we denote by $\phi_{\zz} \in \Pkc_{1}(\omegazh) \cap H^1_0(\omegaz)$ the Lagrange basis function associated with the node $\zz$.
Then for any $\qq_{\zz} \in [\Pkc_{k-1}(\omegazh)]^3$ 
the function $\BB_{\zz} := h_{\zz}^2 \phi_{\zz} \qq_{\zz} \in [\Pkc_{k}(\omegazh)]^3 \cap \HH^1_0(\omegaz)$ satisfies 
\begin{equation}
\label{eq:int-local-inf-sup}
\normai{\BB_{\zz}}^2_{\zz}  
\lesssim \normai{\qq_{\zz}}^2_{*, \zz}
\qquad \text{and} \qquad
(\BB_{\zz}, \, \qq_{\zz})_{\omegaz} 
\gtrsim \normai{\qq_{\zz}}^2_{*, \zz} \,.
\end{equation}
The proof easily follows from \textbf{(MP1)}.
The following observation will be used only for the special case $k=1$. Let $M \in \widetilde{\Omega}_h$ and let $\zz \in \mathcal{N}_h$ be such that $\omegaz \subseteq M$ (cf. Assumption \textbf{(MA2))}.
Then for any $\qq_{M} \in [\Pk_0(M)]^3$, the function $\BB_{M} =\BB_{\zz}=h_{\zz}^2 \phi_{\zz} \qq_{M} $ in $[\Pkc_{k}(M)]^3 \cap \HH^1_0(M)$
(extending $\BB_{\zz}$ to zero in $M \setminus \omegaz$), satisfies
\begin{equation}
\label{eq:int-local-inf-sup1}
\normai{\BB_{M}}^2_{M} 
\lesssim \normai{\qq_{M}}^2_{*, M}
\qquad \text{and} \qquad
(\BB_{M}, \, \qq_{M})_{M} 
\gtrsim \normai{\qq_{M}}^2_{*, M} \,.
\end{equation}
The proof is a direct consequence of \eqref{eq:int-local-inf-sup} and \textbf{(MP2)}.

\noindent 
\textsl{Step 2: Global inf-sup stability.}
We now prove the following global inf-sup stability: for any $\qq \in [\mathbb{O}_{k-1}(\Omega_h)]^3$ there exists $\BB \in [\Pkc_{k}(\Omega_h)]^3 \cap \HH^1_0(\Omega)$ such that
\begin{equation}
\label{eq:int-global-inf-sup}
\normai{\BB}^2  
\lesssim \normai{\qq}^2_{*}
\qquad \text{and} \qquad
(\BB, \, \qq) 
\gtrsim \normai{\qq}^2_{*} \,.
\end{equation}
For $k=1$, for any $\qq \in [\Pk_0(\widetilde{\Omega}_h)]^3$, we set $\qq_{M} := \qq|_{M}$ and consider the function $\BB_{M} \in [\Pkc_{1}(\Omega_h)]^3 \cap \HH^1_0(\Omega)$ that satisfies \eqref{eq:int-local-inf-sup1} (extended to 0 in $\Omega \setminus M$). 
Then the function $\BB := \sum_{M \in \widetilde{\Omega}_h} \BB_{M} \in [\Pkc_{1}(\Omega_h)]^3 \cap \HH^1_0(\Omega)$ clearly satisfies \eqref{eq:int-global-inf-sup}.

\noindent 
For $k>1$, let $\qq \in [\Pkc_{k-1}(\Omega_h)]^3$.
For any $\zz \in \mathcal{N}_h$ let  $\qq_{\zz} := \qq|_{\omegaz}$ and let $\BB_{\zz} \in [\Pkc_{k}(\Omega_h)]^3 \cap \HH^1_0(\omegaz)$ be the function defined above that satisfies \eqref{eq:int-local-inf-sup} (assumed extended to 0 in $\Omega \setminus \omegaz$). We set $\BB := \sum_{\zz \in \mathcal{N}_h} \BB_{\zz} \in [\Pkc_{k}(\Omega_h)]^3 \cap \HH^1_0(\Omega)$ and we prove that satisfies \eqref{eq:int-global-inf-sup}. Indeed from \eqref{eq:int-local-inf-sup} and \textbf{(MP1)} we infer
\[
\begin{aligned}
&\normai{\BB}^2  = \sum_{E \in \Omega_h} h_E^{-2} \Vert \BB \Vert^2_E
\lesssim \sum_{E \in \Omega_h} \sum_{\zz \text{vertex of $E$}}  h_E^{-2} \Vert \BB_{\zz} \Vert^2_E 
=\sum_{\zz \in \mathcal{N}_h} \sum_{E \in \omegazh} h_E^{-2} \Vert \BB_{\zz} \Vert^2_E 
\\
& = \sum_{\zz \in \mathcal{N}_h} \normai{\BB_{\zz}}^2_{\zz}
\lesssim  
\sum_{\zz \in \mathcal{N}_h} \normai{\qq_{\zz}}^2_{*,\zz}
= 
\sum_{\zz \in \mathcal{N}_h} \sum_{E \in \omegazh} h_E^2 \Vert \qq \Vert^2_E 
\lesssim \sum_{E \in \Omega_h} h_E^2 \Vert \qq \Vert^2_E
= \normai{\qq}_*^2 \,,
\\
&(\BB, \, \qq) = 
\sum_{\zz \in \mathcal{N}_h} (\BB_{\zz}, \, \qq) = 
\sum_{\zz \in \mathcal{N}_h} (\BB_{\zz}, \, \qq_{\zz})_{\omegaz} 
\gtrsim \sum_{\zz \in \mathcal{N}_h} \normai{\qq_{\zz}}^2_{*, \zz}
\\
 & = \sum_{\zz \in \mathcal{N}_h} \sum_{E \in \omegazh} h_E^2 \Vert \qq \Vert^2_E 
\gtrsim \sum_{E \in \Omega_h} h_E^2 \Vert \qq \Vert^2_E = 
\normai{\qq}_*^2 \,.
\end{aligned} 
\]

\noindent 
\textsl{Step 3: A projection operator.}
Let us consider the  projection operator
$\Pi \colon \WWc \to [\Pkc_{k}(\Omega_h)]^3 \cap \HH^1_0(\Omega)$ defined  for any $\HH \in \WWc$ via the following mixed problem (where the auxiliary variable $\pp$ is sought in $[\mathbb{O}_{k-1}(\Omega_h)]^3$)
\begin{equation}
\label{eq:int-proj-op}
\left \{
\begin{aligned}
\inpr{\Pi \HH, \, \BB} + (\pp, \, \BB) &= 0 &\quad 
&\text{for all $\BB \in [\Pkc_{k}(\Omega_h)]^3 \cap \HH^1_0(\Omega)$,}
\\
(\Pi \HH, \, \qq) &= (\HH, \, \qq) &\quad 
&\text{for all $\qq \in [\mathbb{O}_{k-1}(\Omega_h)]^3$.}
\end{aligned}
\right.
\end{equation}
Notice that, in the light of the inf-sup stability \eqref{eq:int-global-inf-sup}, the mixed problem above is well-posed, moreover the following stability estimate holds:
\begin{equation}
\label{eq:int-stab}
\begin{aligned}
\normai{\Pi \HH} & \lesssim 
\sup_{\qq \in [\mathbb{O}_{k-1}(\Omega_h)]^3} 
\frac{(\HH, \, \qq)}{\normai{\qq}_{*}} \le
\sup_{\qq \in [\mathbb{O}_{k-1}(\Omega_h)]^3} 
\frac{\sum_{E \in \Omega_h} h_E^{-1} \Vert \HH\Vert_E \, h_E \Vert \qq\Vert_E}
{\normai{\qq}_{*}}
\\
&\leq 
\sup_{\qq \in [\mathbb{O}_{k-1}(\Omega_h)]^3}  
\frac{\normai{\HH} \normai{\qq}_{*}}
{\normai{\qq}_{*}} = \normai{\HH} \,.
\end{aligned}
\end{equation}

\noindent 
\textsl{Step 4: Interpolation operator, orthogonality property and interpolation error estimates.}
We define the interpolation operator 
$\PWWd \colon \WWc \to \WWd$, 
given for all $\HH \in \WWc$ by
\begin{equation}
\label{eq:PWWd}
\HHi := \HH_{\rm SZ} + \Pi(\HH - \HH_{\rm SZ})
\end{equation}
where  $\HH_{\rm SZ} \in \WWd$ is the Scott-Zhang interpolator of $\HH$ (see \cite{Scott-Zhang}).  
We recall that the Scott-Zhang interpolatation operator satisfies the following estimate
\begin{equation}
\label{eq:scott-zhang}
\sum_{E \in \Omega_h} \Big( h_E^{-2} \Vert \HH - \HH_{\rm SZ}\Vert_E^2 + \Vert \bnabla(\HH - \HH_{\rm SZ})\Vert_E^2 \Big)
\lesssim h^{2s}\vert \HH\vert_{s+1,\Omega_h}^2 \qquad \text{for all $E \in \Omega_h$.}
\end{equation}
The orthogonality condition \eqref{eq:int-orth} easily follows from the second equation in \eqref{eq:int-proj-op}:
\[
\begin{aligned}
(\HH - \HHi, \qq_{k-1}) &= ((I - \Pi)(\HH - \HH_{\rm SZ}),  \qq_{k-1}) = 0
\quad \text{for all $\qq_{k-1} \in [\mathbb{O}_{k-1}(\Omega_h)]^3$.}
\end{aligned}
\]
For what concerns the interpolation error estimates \eqref{eq:err-int} for $\alpha=0,1,2$ we have
\[
\begin{aligned}
\sum_{E \in \Omega_h} &h_E^{-\alpha} \Vert \HH - \HHi \Vert_E^2 
= \sum_{E \in \Omega_h} h_E^{2-\alpha} h_E^{-2} \Vert \HH - \HHi \Vert_E^2 
\lesssim
 h^{2-\alpha} \normai{\HH - \HHi}^2 
\\ 
& 
\begin{aligned}
&\lesssim
h^{2-\alpha} \normai{\HH - \HH_{\rm SZ}}^2 + 
h^{2-\alpha} \normai{\Pi(\HH - \HH_{\rm SZ})}^2
&\quad  &\text{(\eqref{eq:PWWd} \& tri. ineq.)}
\\
& \lesssim 
h^{2-\alpha} \normai{\HH - \HH_{\rm SZ}}^2 = 
h^{2-\alpha} \sum_{E \in \Omega_h} h_E^{-2} \Vert \HH - \HH_{\rm SZ} \Vert_E^2
&\quad  &\text{(bound \eqref{eq:int-stab})}
\\
& \lesssim h^{2s+2-\alpha} \vert \HH \vert_{s+1,\Omega_h}^2 
&\quad  &\text{(bound \eqref{eq:scott-zhang})}
\end{aligned} 
\end{aligned}
\]
and
\[
\begin{aligned}
&\Vert \bnabla(\HH - \HHi) \Vert^2 
\lesssim 
\Vert \bnabla(\HH - \HH_{\rm SZ}) \Vert^2 + 
\sum_{E \in \Omega_h} \Vert \bnabla \Pi(\HH - \HH_{\rm SZ}) \Vert_E^2
\\
&\begin{aligned}
& \lesssim 
\Vert \bnabla(\HH - \HH_{\rm SZ}) \Vert^2 + 
\normai{\Pi(\HH - \HH_{\rm SZ})}^2
&\,  &\text{(by Lemma \ref{lm:inverse})}
\\
& \lesssim 
\Vert \bnabla(\HH - \HH_{\rm SZ}) \Vert^2 + 
\normai{\HH - \HH_{\rm SZ}}^2
&\,  &\text{(bound \eqref{eq:int-stab})}
\\
& = 
\Vert \bnabla(\HH - \HH_{\rm SZ}) \Vert^2 +
\sum_{E \in \Omega_h} h_E^{-2} \Vert \HH - \HH_{\rm SZ} \Vert_E^2 
\lesssim 
h^{2s} \vert \HH \vert_{s+1,\Omega_h}^2  
&\,  &\text{(bound \eqref{eq:scott-zhang}).}
\end{aligned}
\end{aligned} 
\]
\end{proof}

\begin{remark}
The local nature of the interpolant ${\cal I}_{\bf W}$ is expressed in \eqref{eq:err-int} by the negative power of $h_E$ in the left hand side. Such bound will allow us to handle certain approximation terms without the need to require an uniform bound on $h/h_E$ (that is, a quasi-uniformity of the mesh). 
\end{remark}

\begin{remark}
Note that the case $k=1$ is somehow different, and needs an additional light assumption on the mesh. The main reason is that, by multiplication with piecewise ${\numberset{P}}_1$ and continuous ``hat'' functions, we can show that $\Pkc_k(\Omega_h)$ is able to satisfy condition \eqref{eq:int-orth} with respect to $\Pkc_{k-1}(\Omega_h)$, with a local norm bound. In the case $k=1$, such result would become useless since $\Pkc_{0}(\Omega_h)$ is the global constant function and carries no asymptotic approximation properties for vanishing $h$. Therefore, in the case $k=1$ we need a slightly modified argument which allows to show orthogonality with respect to piecewise constant functions on a suitable coarser mesh. 
\end{remark}

\subsection{Discrete forms}
\label{sub:forms}

In the present section we define the discrete forms at the basis of the proposed stabilized scheme.
We preliminary make the following assumption on the exact velocity solution $\uu$ and on the advective magnetic field $\TT$.

\smallskip
\noindent
\textbf{(RA1) Regularity assumption for the consistency:}

\noindent
Let $\uu \in \ZZc$ be the velocity solution  of Problem \eqref{eq:linear variazionale}, then $\uu \in \VVs := \VVc \cap  \HH^r(\Omega)$ for some $r> 3/2$.
Furthermore we assume that the magnetic advective field $\TT \in \boldsymbol{C}^0(\Omega)$.

\smallskip
\noindent
Recalling that $\cc \in \HH_0(\diver, \Omega)$, we consider the DG counterparts of the continuous forms in \eqref{eq:forme_c1}. Let
$
\ash(\cdot, \cdot) \colon (\VVs \oplus \VVd) \times \VVd \to \R$,
and 
$
c_h(\cdot, \cdot) \colon (\VVs \oplus \VVd) \times \VVd(\Omega_h) \to \R$, 
defined respectively by
\begin{equation}
\label{eq:forme_d}
\begin{aligned}
\ash(\uu,  \vv_h) &:=  (\beps_h(\uu) ,\, \beps_h(\vv_h))
- \sum_{f \in \Edges} (\media{\beps_h(\uu)\nn_f}_f ,\, \jump{\vv_h}_f)_f  +
\\
&- \sum_{f \in \Edges} (\jump{\uu}_f ,\, \media{\beps_h(\vv_h) \nn_f}_f)_f 
+ 
\bdma \sum_{f \in \Edges} h_f^{-1} (\jump{\uu}_f ,\,\jump{\vv_h}_f)_f
\\
c_h(\uu, \vv_h) &:=  (( \bnabla_h \uu ) \, \cc, \, \vv_h )
- \sum_{f \in \EdgesI} ( (\cc \cdot \nn_f) \jump{\uu}_f ,\, \media{\vv_h}_f)_f +
\\
& + 
\bdmc \sum_{f \in \EdgesI} (\vert \cc \cdot \nn_f \vert \jump{\uu}_f, \, \jump{\vv_h}_f )_f
\end{aligned}
\end{equation}
where the penalty parameters $\bdma$ and $\bdmc$ have to be sufficiently large in order to guarantee the coercivity of the form $\ash(\cdot, \cdot)$  and the stability effect in the convection dominated regime due to the upwinding \cite{upwinding, DiPietro, Hdiv3}.

Due to the coupling between fluid-dynamic equation and magnetic equation, in the proposed scheme we also consider an extra stabilizing form (in the spirit of continuous interior penalty \cite{BFH:2006,CIP}) that penalizes the jumps and the gradient jumps along the convective direction $\TT$. 
Let $J_h(\cdot, \cdot) \colon (\VVs \oplus \VVd) \times \VVd(\Omega_h) \to \R$ be the bilinear form defined by
\begin{equation}
\label{eq:J}
\begin{aligned}
J_h(\uu, \vv_h) &:=   \sum_{f \in \EdgesI} \left( \bdmJ 
( \TT \times \jump{\uu}_f , \TT \times \jump{\vv_h}_f )_f +
\bdmJJ h_f^2 
(\jump{\bnabla_h \uu}_f \TT ,\jump{\bnabla_h \vv_h}_f \TT)_f
\right)
\end{aligned}
\end{equation}
where $\bdmJ$ and $\bdmJJ$ are user-dependent (positive) parameters.

Notice that, under the Assumption \textbf{(RA1)}, the discrete forms in \eqref{eq:forme_d} and \eqref{eq:J} satisfy the following consistency property
\begin{equation}
\label{eq:forme_con}
\ash(\uu, \vv_h) = -(\bdiver (\beps (\uu)) ,\, \vv_h )\,,
\,\,\,
c_h(\uu, \vv_h) = c(\uu, \vv_h)\,, 
\,\,\,
J_h(\uu, \vv_h) = 0 \,,
\quad 
\text{$\forall \vv_h \in \VVd$.}
\end{equation}

\begin{remark}
\label{rm:semplificazione}
The following slightly simpler forms 
\begin{equation}
\label{eq:forme_d2}
\begin{aligned}
c_h(\uu, \vv_h) &:=  (( \bnabla_h \uu ) \, \cc, \, \vv_h )
- \sum_{f \in \EdgesI} ( (\cc \cdot \nn_f) \jump{\uu}_f ,\, \media{\vv_h}_f)_f +
\\
& + 
\bdmc \sum_{f \in \EdgesI} \Vert\cc \cdot \nn_f \Vert_{L^\infty(f)} 
( \jump{\uu}_f , \jump{\vv_h}_f )_f
\\
J_h(\uu, \vv_h) &:=  \sum_{f \in \EdgesI} 
\Vert \TT\Vert_{L^\infty(f)}^2 
\left( \bdmJ (\jump{\uu}_f ,  \jump{\vv_h}_f)_f + 
\bdmJJ  h_f^2 (\jump{\bnabla_h \uu}_f  , \jump{\bnabla_h \vv_h}_f )_f
\right)
\end{aligned}
\end{equation}
can be adopted in the place of the forms in \eqref{eq:forme_d} and \eqref{eq:J}, respectively. The simplified forms still satisfy the consistency properties \eqref{eq:forme_con}.
The theoretical derivations of Section \ref{sec:theo} trivially extend also to these forms by changing the semi-norms $\normaupw{\cdot}$ and $\normacip{\cdot}$ in \eqref{eq:norme} accordingly. 
\end{remark}



\subsection{Discrete scheme}
\label{sub:scheme}

Referring to the spaces \eqref{eq:spazi_d}, the forms \eqref{eq:forme_c2}, \eqref{eq:forme_d} and \eqref{eq:J}, the stabilized
method for the MHD equation is given by

\begin{equation}
\label{eq:linear fem}
\left\{
\begin{aligned}
& \text{find $(\uu_h, p_h, \BB_h) \in \VVd \times \Qd \times \WWd$,  such that} 
\\
&\begin{aligned}
\pss (\uu_h, \vv_h) + \ns \ash(\uu_h, \vv_h) + c_h(\uu_h, \vv_h) + &
\\ 
-d(\BB_h, \vv_h)   + J_h(\uu_h, \vv_h) + b(\vv_h, p_h) 
&= 
(\ff, \vv_h)  
&\,\,\, & \text{for all $\vv_h \in \VVd$,} 
\\
b(\uu_h, q_h) &= 0 
&\,\,\, & \text{for all $q_h \in \Qd$,}
\\
\psm (\BB_h, \HH_h) + \nm \am(\BB_h, \HH_h) + d(\HH_h, \uu_h) 
&= 
(\GG, \HH_h)  
&\,\,\, & \text{for all $\HH_h \in \WWd$.}
\end{aligned}
\end{aligned}
\right.
\end{equation}
%

\section{Theoretical analysis}
\label{sec:theo}

\subsection{Stability analysis}
\label{sub:stab}

Let $\ZZs := \ZZc \cap \HH^r(\Omega)$ (cf. Assumption \textbf{(RA1)}), then recalling the definition \eqref{eq:Z_h}, consider the form
$\astab \colon 
(\ZZs \oplus \ZZd) \times \WWc  \times \ZZd \times \WWd \to \R$
defined by
\begin{equation}
\label{eq:astab}
\begin{aligned}
\astab(\uu, \BB, \vv_h, \HH_h) := &
\pss (\uu, \vv_h) + \psm (\BB, \HH_h) + 
\ns \ash(\uu, \vv_h) + \nm \am(\BB, \HH_h) +
\\  
& 
+ c_h(\uu, \vv_h) -d(\BB, \vv_h)   + d(\HH_h, \uu) + J_h(\uu, \vv_h) \,.
\end{aligned}
\end{equation}
Then Problem \eqref{eq:linear fem} can be formulated as follows
\begin{equation}
\label{eq:linear stab fem}
\left\{
\begin{aligned}
& \text{find $(\uu_h, \BB_h) \in \ZZd \times \WWd$,  such that} 
\\
&\begin{aligned}
\astab(\uu_h, \BB_h, \vv_h, \HH_h) 
&= 
(\ff, \vv_h)   + (\GG, \HH_h)  
&\,\,\, & \text{for all $\vv_h \in \ZZd$, $\HH_h \in \WWd$.}
\end{aligned}
\end{aligned}
\right.
\end{equation}
Notice that under Assumption \textbf{(RA1)}, employing \eqref{eq:forme_con}, the form $\astab(\cdot, \cdot, \cdot, \cdot)$ is consistent, i.e. the solution $(\uu, \BB) \in \ZZs \times \WWc$ of Problem \eqref{eq:linear variazionale} realizes
\begin{equation}
\label{eq:linear stab}
\astab(\uu, \BB, \vv_h, \HH_h) = 
(\ff, \vv_h)   + (\GG, \HH_h)  
\qquad \text{for all $\vv_h \in \ZZd$, $\HH_h \in \WWd$.}
\end{equation}
We define the following norm and semi-norm  on $\VVs \oplus \VVd$
\begin{equation}
\label{eq:norme}
\begin{aligned}
\normas{\uu}^2 &:= 
\pss \Vert \uu \Vert^2 + \ns \Vert \beps_h(\uu) \Vert^2 + 
\ns \bdma\sum_{f \in \Edges} h_f^{-1} \Vert \jump{\uu}_f \Vert_{f}^2
\\
\normaupw{\uu}^2 &:= \bdmc\sum_{f \in \EdgesI}  \Vert \vert \cc \cdot \nn_f \vert^{1/2} \jump{\uu}_f \Vert_{f}^2
\\
\normacip{\uu}^2 &:=
\bdmJ\sum_{f \in \EdgesI}  \Vert \TT \times \jump{\uu}_f \Vert_{f}^2
+
\bdmJJ\sum_{f \in \EdgesI}  h_f^2 \Vert \jump{\bnabla_h \uu}_f \TT\Vert_{f}^2
\end{aligned}
\end{equation}
and the energy norms on $\VVs \oplus \VVd$  and $\WWc$ respectively
\begin{equation}
\label{eq:normeVW}
\normastab{\uu}^2 := 
\normas{\uu}^2 + \normaupw{\uu}^2 + \normacip{\uu}^2 \,,
\qquad 
\normam{\BB}^2 := 
\psm \Vert \BB \Vert^2 + \nm \Vert \bnabla \BB \Vert^2 \,.
\end{equation}
Finally, we define the following mesh-dependent norm on $\HH^1(\Omega_h)$
\begin{equation}
\label{eq:norma1h}
\Vert \uu \Vert_{1,h}^2 := 
\Vert \uu \Vert^2 + \Vert \beps_h(\uu) \Vert^2 + 
\bdma\sum_{f \in \Edges}  h_f^{-1} \Vert \jump{\uu}_f \Vert_{f}^2 \,.
\end{equation}
Then from \cite{korn} we recall the following.

\begin{lemma}[Discrete Korn inequality]
\label{lm:korn}
Under the mesh assumption \cfan{(MA1)}, for any $\vv \in \HH^1(\Omega_h)$ it holds 
\[
\Vert \bnabla_h \vv \Vert^2 \lesssim 
\Vert \vv \Vert_{1,h}^2 \,.
\]
\end{lemma}

The following result are instrumental to prove the well-posedness of problem \eqref{eq:linear stab fem}.

\begin{proposition}[coercivity of $\astab$]
\label{prp:coe}
Under the mesh assumption \cfan{(MA1)} and the consistency assumption \cfan{(RA1)},
let $\astab(\cdot, \cdot, \cdot, \cdot)$ be the form \eqref{eq:astab}.
If the parameter $\bdma$ in \eqref{eq:forme_d} is sufficiently large there exists a real positive constant $\ccoe$ such that for all $\vv_h \in \ZZd$ and $\HH_h \in \WWc$ the following holds
\[
\astab(\vv_h, \HH_h, \vv_h, \HH_h) \geq 
\ccoe \left( \normastab{\vv_h}^2 + \normam{\HH_h}^2 \right) \,.
\]
\end{proposition}

\begin{proof}
\begin{proof}
We briefly sketch the simple proof. Recalling definitions \eqref{eq:norme} and \eqref{eq:normeVW}, direct computations yield
\begin{equation}
\label{eq:coe1}
\begin{aligned}
\astab(\vv_h, \HH_h, \vv_h, \HH_h) = &
\pss \Vert \vv_h \Vert^2 +  \ns \ash(\vv_h, \vv_h) 
+ c_h(\vv_h, \vv_h) + \normacip{\vv_h}^2 +\\
& +
\psm \Vert \HH_h \Vert^2 + \nm \Vert \bcurl (\HH_h) \Vert^2 
+ \nm \Vert \diver \HH_h \Vert^2 \,.
\end{aligned}
\end{equation}
From \cite[Lemma 4.12, Lemma 2.18]{DiPietro} we infer
\begin{equation}
\label{eq:coe2}
\begin{aligned}
\pss \Vert \vv_h \Vert^2 +  \ns \ash(\vv_h, \vv_h)   &\gtrsim 
\normas{\vv_h}^2 \,,
\qquad
c_h(\vv_h, \vv_h) = \normaupw{\vv_h}^2 \,.
\end{aligned}
\end{equation}
Furthermore being $\Omega$ a convex polyhedron from \cite[Theorem 3.9]{Girault-book} and \cite{Amrouche} we infer
\begin{equation}
\label{eq:coe3}
\Vert \bcurl(\HH_h)\Vert^2 +  \Vert \diver \HH_h \Vert^2 
\gtrsim  \Vert \bnabla \HH_h \Vert^2 \,.
\end{equation} 
The thesis follows inserting \eqref{eq:coe2} and \eqref{eq:coe3} in \eqref{eq:coe1}.
\end{proof}
\end{proof}

\subsection{Error analysis}
\label{sub:err}

Let $(\uu, \BB) \in \ZZs \times \WWc$ and $(\uu_h, \BB_h) \in \ZZd \times \WWd$ be the solutions of Problem \eqref{eq:linear variazionale} and Problem \eqref{eq:linear stab fem}, respectively.
Then referring to Lemma \ref{lm:int-v} and Lemma \ref{lm:int-i}, let us define the following error functions
\begin{equation}
\label{eq:fquantities}
\eei := \uu - \uui\,, \qquad 
\eeh := \uu_h - \uui\,, \qquad 
\EEi := \BB - \BBi\,, \qquad 
\EEh := \BB_h - \BBi \,.
\end{equation}
Notice  that from Lemma \ref{lm:int-v} $\uui \in \ZZd$, then $\eeh \in \ZZd$.
We also introduce the following useful quantities for the error analysis
\begin{equation}
\label{eq:quantities}
\begin{aligned}
\ls^2 &:= \max \biggl\{ 
\pss h^2 \,,  \max_{f \in \EdgesI} \Vert  \cc \cdot \nn_f \Vert_{L^\infty(f)} h \,,
 \max_{f \in \EdgesI} \Vert  \TT \Vert_{L^{\infty}(f)}^2  h \,,
\ns(1 + \bdma + \bdma^{-1}) 
\biggr\} \,,
\\
\lm^2 &:= \max \{ \psm h^2 \,, \nm \} \,.
\end{aligned}
\end{equation}
Notice that the quantities appearing in $\ls$ and $\lm$ are those typically used to determine the global regime of the MHD equation i.e. diffusion, convective or reaction dominated case.

We now state the final regularity assumptions required for the theoretical analysis.

\smallskip
\noindent
\textbf{(RA2) Regularity assumptions on the exact solution (error analysis):} 

\noindent
Assume that:
(i) the advective velocity field $\cc \in \boldsymbol{W}^1_ \infty(\Omega)$,
(ii) the advective magnetic induction $\TT \in \boldsymbol{W}^1_ \infty(\Omega)$,
(iii) the exact solution $(\uu, p, \BB)$ of Problem \eqref{eq:linear variazionale} belongs to $(\HH^{k+1}(\Omega_h) \cap \HH^r(\Omega)) \times H^{k}(\Omega_h) \times \HH^{k+1}(\Omega_h)$ for some $r>3/2$.

\begin{proposition}[Interpolation error estimate in the energy norms]
Let Assumption \cfan{(MA1)} hold.
Furthermore, if $k=1$ let also Assumption \cfan{(MA2)} hold. Then, under the regularity assumption  \cfan{(RA2)}, referring to \eqref{eq:fquantities} and \eqref{eq:quantities} the following hold:
\label{prp:interpolation}
\[
\normastab{\eei}^2 \lesssim  
\ls^2 \, h^{2k} \vert \uu \vert_{k+1, \Omega_h}^2 \,,
\qquad
 \normam{\EEi}^2 \lesssim 
\lm^2 \, h^{2k} \vert \BB \vert_{k+1, \Omega_h}^2 \,.
\]
\end{proposition}

\begin{proof}
The proof is a direct consequence of bounds \eqref{eq:int-v}, \eqref{eq:utilejump2} and \eqref{eq:err-int}:
\begin{align}
\label{eq:eiS}
\normas{\eei}^2 
& \lesssim
(\pss h^2 + \ns(1 + \bdma)) h^{2k} \vert \uu \vert_{k+1, \Omega_h}^2\,,
\\
\label{eq:eiupw}
\normaupw{\eei}^2 
& \lesssim
\bdmc \max_{f \in \EdgesI} \Vert \cc \cdot \nn_f \Vert_{L^\infty(f)}
h^{2k+1} \vert \uu\vert_{k+1, \Omega_h}^2 \,,
\\
\label{eq:eicip}
\normacip{\eei}^2 
& \lesssim
(\bdmJ + \bdmJJ) \max_{f \in \EdgesI} \Vert \TT \Vert_{L^\infty(f)}^2 
h^{2k+1}\vert \uu \vert_{k+1, \Omega_h}^2 \,,
\\
\label{eq:EiM}
\normam{\EEi}^2 & \lesssim
(\psm h^2 + \nm) h^{2k} \vert \BB \vert_{k+1, \Omega_h}^2 \,.
\end{align}
\end{proof}

In order to prove the convergence result in Proposition \ref{prp:error equation}, we need the following technical lemmas.

\begin{lemma}
\label{lm:norma1h}
Let Assumption \cfan{(MA1)} hold. Then for any $\vv_h \in \VVd$ 
\[
\normas{\vv_h}^2 \lesssim \ts^2 \Vert \vv_h\Vert_{1,h}^2 \,,
\qquad 
\Vert \bnabla_h \vv_h\Vert^2 \lesssim \Vert \vv_h\Vert_{1,h}^2 
\lesssim \ps^2 \normas{\vv_h}^2 \,,
\]
where
\begin{equation}
\label{eq:psi}
\ts^2 := \max\{\pss, \ns\} \,,
\qquad
\ps^2 := \max\{\pss^{-1},  \ns^{-1} \} \,.
\end{equation}
\end{lemma}

\begin{proof}
The bounds easily follow from definitions \eqref{eq:norme} and \eqref{eq:norma1h} and from the discrete Korn inequality in Lemma \ref{lm:korn}.
\end{proof}

\begin{remark}
\label{rm:utile-min}
Notice that from Lemma \ref{lm:inverse} and Lemma \ref{lm:norma1h}, for any $\vv_h \in \VVd$ we infer
\begin{equation}
\label{eq:utile-min}
\sum_{E \in \Omega_h} h_E^2 \Vert \bnabla \vv_h \Vert^2_E
\lesssim \min \{\pss^{-1}, \ps^2  h^2 \} \normas{\vv_h}^2 \,.
\end{equation}
\end{remark}

\begin{lemma}
\label{lm:cip}
Let Assumption \cfan{(MA1)} hold. Furthermore, if $k=1$ let also Assumption \cfan{(MA2)} hold. Then, referring to \eqref{eq:mathbbO}, there exists a projection operator $\intcip \colon \Pk_{k-1}(\Omega_h) \to \mathbb{O}_{k-1}(\Omega_h)$
such that for any $p_{k-1} \in \Pk_{k-1}(\Omega_h)$   the following holds:
\[
\sum_{E \in \Omega_h}h_E  \Vert (I- \intcip) p_{k-1} \Vert_E^2 \lesssim 
\sum_{f \in \EdgesI}
h_f^2 \Vert \jump{p_{k-1}}_f\Vert_f^2
\, .
\]
\end{lemma}

\begin{proof}

For $k>1$, let  $\intcip \colon \Pk_{k-1}(\Omega_h) \to \Pkc_{k-1}(\Omega_h)$ be the Oswald interpolation  operator \cite{Hoppe}.
Then the desired bound was proved in \cite[Lemma 3.2]{CIP}.

For $k=1$, let  $\intcip \colon \Pk_0(\Omega_h) \to \Pk_0(\widetilde{\Omega}_h)$ be the $L^2$-projection operator
\begin{equation}
\label{eq:intcipo0}
( q_0 , \, p_0 - \, \intcip p_0 ) = 0 \qquad  \text{for all $p_0 \in \Pk_0(\Omega_h)$  and $q_0 \in \Pk_0(\widetilde{\Omega}_h)$.} 
\end{equation}
Referring to Assumption \textbf{(MA2)}, for any $M \in \widetilde{\Omega}_h$, we introduce the following notations:
$\boldsymbol{V}_{2,h}(M) := [\Pk_2(\omegaEh)]^3 \cap \boldsymbol{H}^1_0(M)$,
and $Q_{2,h}(M) := \Pk_0(\omegaEh) \cap L^2_0(M)$.
Furthermore we define  the set of internal faces associated to $\omegaEh$
\[
\EdgesI(M) := \{ f \in \Edges \,\, \text{s.t.} \,\, f \in \EdgesE \,\, \text{for some $E \in \omegaEh$ and} \,\, (f \cap \partial M)^\circ = \emptyset \} \,.
\]

By definition \eqref{eq:intcipo0}, 
$(p_0 - \, \intcip p_0)|_{M} \in Q_{2,h}(M)$ for any $M \in \widetilde{\Omega}_h$. Employing the inf-sup stability of the pair  $(\boldsymbol{V}_{2,h}(M), Q_{2,h}(M))$  (see \cite{boffi-brezzi-fortin:book}) we infer

\[
\begin{aligned}
&\Vert (I - \intcip) p_0 \Vert_{M} 
\lesssim
\sup_{\vv_h \in \boldsymbol{V}_{2,h}(M)}
\left(p_0 - \intcip p_0, \, \diver \vv_h \right)_{M}/\Vert \bnabla \vv_h \Vert_{M}
\\
&\begin{aligned}
& =
\sup_{\vv_h \in \boldsymbol{V}_{2,h}(M)}
\sum_{E \in \omegaEh} \left(p_0 - \intcip p_0,\,  \vv_h \cdot \nn_E\right)_{\partial E}/\Vert \bnabla \vv_h \Vert_{M}
&  & \text{(int. by parts)}
\\
& =
\sup_{\vv_h \in \boldsymbol{V}_{2,h}(M)}
\sum_{f \in \EdgesI(M)}
\left(\jump{p_0}_f , \, \vv_h \cdot \nn_f\right)_f/\Vert \bnabla \vv_h \Vert_{M}
&  & \text{(${\vv_h}|_{\partial M}=\boldsymbol{0}$).}
\end{aligned}
\end{aligned}
\]
Then by Cauchy-Schwarz inequality
\begin{equation}
\label{eq:intcipk0}
\Vert (I - \intcip) p_0 \Vert_{M}^2  
\lesssim
\biggl(\sum_{f \in \EdgesI(M)}
h_f \Vert \jump{p_0}_f \Vert_f^2 \biggr) 
\sup_{\vv_h \in \boldsymbol{V}_{2,h}(M)}
\sum_{f \in \EdgesI(M)}
h_f^{-1} \Vert \vv_h \Vert_f^2
/\Vert \bnabla \vv_h \Vert_{M}^2.
\end{equation}
We now prove that the $\sup$ is uniformly bounded.  From \eqref{eq:utile-jump} (with $\alpha=-1$ and $\EdgesI(M)$ and $\omegaEh$ in the place of $\Edges$ and $\Omega_h$ respectively), mesh assumption \textbf{(MP2)}, and the scaled Poincar\'e inequality,  for any $\vv_h \in \boldsymbol{V}_{2,h}(M)$ we infer
\begin{equation}
\label{eq:intcipk1}
\sum_{f \in \EdgesI(M)}
h_f^{-1} \Vert \vv_h \Vert_f^2 \lesssim \sum_{E \in \omegaEh} h_E^{-2} \Vert \vv_h\Vert_E^2
\lesssim h_{M}^{-2} \Vert \vv_h\Vert_{M}^2 \lesssim \Vert \bnabla \vv_h \Vert_{M}^2 \,.
\end{equation}
Then collecting \eqref{eq:intcipk0} and \eqref{eq:intcipk1} and and recalling  mesh properties \textbf{(MP1)} and \textbf{(MP2)} we have
\begin{equation}
\label{eq:intcipk2}
h_{M}  \Vert (I - \intcip) p_0 \Vert_{M}^2  \lesssim 
\sum_{f \in \EdgesI(M)}
h_f^2 \Vert \jump{p_0}_f \Vert_f^2 \,.
\end{equation}
The thesis now follows adding the local bounds \eqref{eq:intcipk2} 
\[
\begin{aligned}
\sum_{E \in \Omega_h} h_E \Vert (I - \intcip) p_0 \Vert_{E}^2 
&\lesssim 
\sum_{M \in \omegaEh} h_{M}  
\Vert (I - \intcip) p_0 \Vert_{M}^2 
\lesssim 
\sum_{f \in \EdgesI}
h_f^2 \Vert \jump{p_0}_f \Vert_f^2 \,.
\end{aligned}
\]
\end{proof}

\begin{remark}
In the case $k=1$ we have piecewise constant functions and we need to control them with the jumps across faces, which is quite natural. On the other hand, in order to extract the correct scalings in the mesh size we need to apply the associated Poincar\'e inequality on $h$-sized domains (and not on the whole $\Omega$), which is where the presence of the $\intcip$ becomes critical.
\end{remark}

We now prove the following error estimation.
\begin{proposition}[Discretization error]
\label{prp:error equation}
Let Assumption \cfan{(MA1)} hold. Furthermore, if $k=1$ let also Assumption \cfan{(MA2)} hold.
Then, under the regularity assumption \cfan{(RA2)} and assuming that the parameter $\bdma$ (cf. \eqref{eq:forme_d}) is sufficiently large, referring to \eqref{eq:fquantities}, \eqref{eq:quantities} and \eqref{eq:psi}, the following holds
\label{eq:error equation}
\[
\begin{aligned}
\normastab{\eeh}^2 + \normam{\EEh}^2
\lesssim 
(\ls^2 + \gs^2 + \gm^2) h^{2k} \vert \uu \vert_{k+1, \Omega_h}^2 + 
(\lm^2 + \gs^2) h^{2k} \vert \BB \vert_{k+1, \Omega_h}^2 \,,
\end{aligned}
\]
where
\begin{equation}
\label{eq:gamma}
\begin{aligned}
\gs^2 &:= \min\{ \pss^{-1} h^2,  \, \ps^2 h^4\}
\Vert\cc\Vert_{W^{1,\infty}(\Omega_h)}^2 +
\pss^{-1} h^2 \Vert\TT\Vert_{W^{1,\infty}(\Omega_h)}^2 + h\,,
\\
\gm^2 &:=
\min\{ \psm^{-1} h^2,  \, \nm^{-1} h^4\} 
\Vert\TT\Vert_{W^{1,\infty}(\Omega_h)}^2 \,.
\end{aligned}
\end{equation}
\end{proposition}

\begin{proof}
Employing Proposition \ref{prp:coe}, Problem \eqref{eq:linear stab fem}, equation \eqref{eq:linear stab} and definition \eqref{eq:astab} we get
\begin{equation}
\label{eq:conv0}
\begin{aligned}
\normastab{\eeh}^2 &+ \normam{\EEh}^2  \lesssim
\astab(\eeh, \EEh, \eeh, \EEh) 
\\
&=
\astab(\uu_h, \BB_h, \eeh, \EEh) -
\astab(\uui, \BBi, \eeh, \EEh) 
\\
& =
\astab(\uu, \BB, \eeh, \EEh) -
\astab(\uui, \BBi, \eeh, \EEh) 
= \astab(\eei, \EEi, \eeh, \EEh)
\\
& =
\bigl( \pss (\eei, \eeh) + \ns \ash(\eei, \eeh) \bigr)
+
\bigl(\psm (\EEi, \EEh) + \nm \am(\EEi, \EEh)  \bigr) 
+
\\  
& 
+ c_h(\eei, \eeh) -d(\EEi, \eeh)   + d(\EEh, \eei) + J_h(\eei, \eeh) 
=: \sum_{i=1}^6 \alpha_i \,.
\end{aligned}
\end{equation}

We estimate separately each term in the sum above.

\noindent
$\bullet$ Estimate of $\alpha_1$: 

\begin{equation}
\label{eq:alpha1}
\begin{aligned}
&\alpha_1 
=
\Big( \pss (\eei,\, \eeh) +
\ns (\beps_h(\eei), \, \beps_h(\eeh))
+
\ns \bdma \sum_{f \in \Edges} 
h_f^{-1} (\jump{\eei}_f ,\,  \jump{\eeh}_f )_f
\Big)
+
\\
& 
-
\ns \sum_{f \in \Edges} ( \media{\beps_h(\eei) \nn_f}_f ,\, \jump{\eeh}_f )_f 
-
\ns \sum_{f \in \Edges} ( \jump{\eei}_f ,\, \media{\beps_h(\eeh) \nn_f}_f )_f 
=: \sum_{j=1}^3 \alpha_{1,j} \,.
\end{aligned}
\end{equation}
Cauchy-Schwarz inequality and \eqref{eq:eiS} yield
\begin{equation}
\label{eq:alpha1-1}
\alpha_{1,1} \leq \normas{\eei}  \normas{\eeh}
\lesssim  \ls \, h^{k} \vert \uu \vert_{k+1, \Omega_h}  \normas{\eeh} \,.
\end{equation}
Employing the Cauchy-Schwarz inequality and  \eqref{eq:utilejump2}, the terms $\alpha_{1,2}$ can be bounded as follows
\begin{equation}
\label{eq:alpha1-2}
\begin{aligned}
\alpha_{1,2} \leq  
\biggl( \ns  \bdma^{-1} \sum_{f \in \Edges} h_f \Vert \media{\beps_h (\eei)}_f \Vert_{f}^2 \biggr)^{1/2} 
\normas{\eeh}
\lesssim
\ls \, h^k \vert \uu \vert_{k+1, \Omega_h} \normas{\eeh} \,.
\end{aligned}
\end{equation}
Whereas for $\alpha_{1,3}$ from \eqref{eq:utilejump2} and \eqref{eq:utile-jump} we infer
\begin{equation}
\label{eq:alpha1-3}
\begin{aligned}
\alpha_{1,3} &\leq \ns   
\biggl( \sum_{f \in \Edges} h_f^{-1} \Vert \jump{\eei}_f \Vert_{f}^2 \biggr)^{1/2}
\biggl(  \sum_{f \in \Edges} h_f \Vert \media{\beps_h (\eeh)}_f \Vert_{f}^2 \biggr)^{1/2}
\\
& \lesssim
\ns h^{k} \vert \uu \vert_{k+1, \Omega_h}  
\Vert \beps_h (\eeh) \Vert
\lesssim
\ls \, h^{k} \vert \uu \vert_{k+1, \Omega_h}
\normas{\eeh} \,.
\end{aligned}
\end{equation}
Therefore collecting \eqref{eq:alpha1-1}, \eqref{eq:alpha1-2} and \eqref{eq:alpha1-3} in \eqref{eq:alpha1} we obtain
\begin{equation}
\label{eq:alpha1-f}
\alpha_1 \lesssim \ls \, h^{k} \vert \uu \vert_{k+1, \Omega_h} \normas{\eeh}
\lesssim \ls \, h^{k} \vert \uu \vert_{k+1, \Omega_h} \normastab{\eeh}
\,.
\end{equation}

\noindent
$\bullet$ Estimate of $\alpha_2$: 
from Cauchy-Schwarz inequality and Proposition \ref{prp:interpolation} we get

\begin{equation}
\label{eq:alpha2-f} 
\begin{aligned}
\alpha_2 &=
\psm (\EEi, \EEh) +
\nm  (\bcurl(\EEi) , \bcurl(\EEh)) +
\nm (\diver(\EEi) , \diver(\EEh) ) 
\\
& \lesssim
\normam{\EEi} \normam{\EEh}
\lesssim 
\lm h^{k}\vert \BB\vert_{k+1,\Omega_h} \normam{\EEh} \,.
\end{aligned}
\end{equation}

\noindent
$\bullet$ Estimate of $\alpha_3$: we preliminary observe that being  
$\diver \, \cc = 0$ and $\cc \cdot \nn = 0$ on $\partial \Omega$, an integration by parts yield 
\begin{equation}
\label{eq:alpha2--2}
\begin{aligned}
&(( \bnabla_h \eei ) \, \cc,  \eeh )
=
- (( \bnabla_h \eeh ) \, \cc,  \eei )
+ \sum_{f \in \EdgesI} \int_f
(\cc \cdot \nn_f) \jump{\eei \cdot \eeh}_f 
\, {\rm d}f \,.
\end{aligned}
\end{equation}
Furthermore for any $f \in \EdgesI$ direct computations give
\begin{equation}
\label{eq:alpha2--1}
 \jump{\eei \cdot \eeh}_f  =
\jump{\eei}_f \cdot \media{\eeh}_f 
+
\media{\eei}_f \cdot \jump{\eeh}_f \,.
\end{equation}
Therefore from \eqref{eq:alpha2--2} and \eqref{eq:alpha2--1},  recalling definition \eqref{eq:forme_d2}, we obtain
\begin{equation}
\label{eq:alpha3}
\begin{aligned}
\alpha_3 &= - (( \bnabla_h \eeh ) \, \cc, \eei )
+ \sum_{f \in \EdgesI} ((\cc \cdot \nn_f) \media{\eei}_f , \jump{\eeh}_f )_f +
\\
& + 
\bdmc \sum_{f \in \EdgesI} ( \vert \cc \cdot \nn_f \vert \jump{\eei}_f , \jump{\eeh}_f )_f
 =: \sum_{j=1}^3 \alpha_{3,j} \,.
\end{aligned}
\end{equation}
Recalling \eqref{eq:orth-v}, for the term $\alpha_{3,1}$ we infer
\begin{equation}
\label{eq:alpha3-1}
\begin{aligned}
\alpha_{3,1} & =
(\eei\,, ( \bnabla_h \eeh ) \, ((\Pzerok{0} - I)\cc)) 
\leq 
\sum_{E \in \Omega_h}
\Vert ( I - \Pzerok{0})\cc\Vert_{L^\infty(E)} \Vert \eei \Vert_{E} \Vert \bnabla_h \eeh \Vert_{E}
\\
& \begin{aligned}
& \lesssim 
\sum_{E \in \Omega_h}
h_E \Vert\cc\Vert_{W^1_\infty(E)} \Vert \eei \Vert_{E} \Vert \bnabla_h \eeh \Vert_{E}
&  & \text{(Lemma \ref{lm:bramble})}
\\
& \lesssim
\Vert\cc\Vert_{W^{1,\infty}(\Omega_h)} \Vert \eei \Vert 
\biggl( \sum_{E \in \Omega_h} h_E^2 \Vert \bnabla_h \eeh \Vert_{E}^2 \biggr)^{1/2}
&  & \text{(Cauchy-Schwarz)}
\\
& 
\lesssim
\min\{ \pss^{-1/2},  \ps h\} 
\Vert\cc\Vert_{W^{1,\infty}(\Omega_h)} h^{k+1} \vert \uu \vert_{k+1, \Omega_h} \normas{\eeh}
&  & 
\text{(Lemma \ref{lm:int-v} \& \eqref{eq:utile-min})}
\\
& 
\lesssim
\gs h^k \vert \uu \vert_{k+1, \Omega_h} \normas{\eeh}
&  & \text{(def. \eqref{eq:gamma})}
\end{aligned}
\end{aligned}
\end{equation}
For the term $\alpha_{3,2}$ from the Cauchy-Schwarz inequality and \eqref{eq:utilejump2} we infer
\begin{equation}
\label{eq:alpha3-2}
\begin{aligned}
\alpha_{3,2} 
&\leq 
\bdmc^{-1/2} \max_{f \in \EdgesI} \Vert \cc \cdot \nn_f \Vert_{L^\infty(f)}^{1/2} 
\bigl( \sum_{f \in \EdgesI} 
\Vert \media{\eei}_f\Vert_f^2
\bigr)^{1/2}
\normaupw{\eeh}
\\
& \lesssim
\max_{f \in \EdgesI} \Vert \cc \cdot \nn_f \Vert_{L^\infty(f)}^{1/2} 
h^{k+1/2} \vert \uu\vert_{k+1, \Omega_h} 
\normaupw{\eeh}
\lesssim
\ls \,
h^{k} \vert \uu\vert_{k+1, \Omega_h} 
\normaupw{\eeh} \,.
\end{aligned}
\end{equation}
Finally for the last term, the Cauchy-Schwarz inequality and \eqref{eq:eiupw} easily imply
\begin{equation}
\label{eq:alpha3-3}
\begin{aligned}
\alpha_{3,3} 
& \leq \normaupw{\eei} \normaupw{\eeh}
\lesssim
\ls \,
h^{k} \vert \uu\vert_{k+1, \Omega_h} 
\normaupw{\eeh} \,.
\end{aligned}
\end{equation}
Therefore from \eqref{eq:alpha3}, \eqref{eq:alpha3-1}, \eqref{eq:alpha3-2} and \eqref{eq:alpha3-3} we have
\begin{equation}
\label{eq:alpha3-f}
\begin{aligned}
\alpha_3 &\lesssim
h^{k} \vert \uu\vert_{k+1, \Omega_h} 
\left(  \gs \, \normas{\eeh}  +
\ls \normaupw{\eeh} \right) 
\lesssim
(\ls + \gs) h^{k} \vert \uu\vert_{k+1, \Omega_h} 
\normastab{\eeh} \,.
\end{aligned}
\end{equation}

\noindent
$\bullet$ Estimate of $\alpha_4$: a vector calculus identity and an integration by parts yield
\begin{equation}
\label{eq:alpha-4}
\begin{aligned}
\alpha_4 & = 
-(\bcurl (\EEi) \times \TT,  \eeh  )
= ( \bcurl (\EEi) , \eeh  \times \TT)
\\
& =
\sum_{E \in \Omega_h} 
(\EEi,  \bcurl_h (\eeh  \times \TT))_E +
\sum_{E \in \Omega_h}
(\EEi \times \nn_E ,   \eeh  \times \TT)_{\partial E}
\\
& = 
(\EEi , \bcurl_h (\eeh  \times ((I - \Pzerok{0})\TT))) +
(\EEi , \bcurl_h (\eeh  \times (\Pzerok{0}\TT))) +
\\
& +
\sum_{f \in \EdgesI}
(\EEi \times \nn_f , \jump{\eeh}  \times \TT)_f 
=: \sum_{j=1}^3 \alpha_{4,j}
\end{aligned}
\end{equation}
where in the summation over the faces we can restrict on the internal faces since, for every boundary face $f$, recalling that $\TT \cdot \nn_f$, $\eeh \cdot \nn_f  =0$ on $\partial \Omega$ it holds:
\[
\begin{aligned}
(\EEi \times \nn_f) \cdot  (\eeh \times \TT) 
& =
(\EEi \cdot \eeh) (\TT \cdot \nn_f) -
(\EEi \cdot \TT) (\eeh \cdot \nn_f) = 0 \,.
\end{aligned}
\]
For $\alpha_{4,1}$ applying the Cauchy-Schwarz inequality we infer
\begin{equation}
\label{eq:alpha4-1}
\begin{aligned}
\alpha_{4,1} & \leq
\sum_{E \in \Omega_h} 
\Vert \EEi\Vert_{E} 
\Vert \bcurl_h (\eeh  \times ((I - \Pzerok{0})\TT)) \Vert_{E} 
\\
& \lesssim
\sum_{E \in \Omega_h} 
\Vert \EEi\Vert_{E} 
\left(
\Vert (I - \Pzerok{0})\TT \Vert_{L^\infty(E)} \Vert \bnabla_h \eeh \Vert_{E} +
\Vert (I - \Pzerok{0})\TT \Vert_{W^{1,\infty}(E)} \Vert \eeh \Vert_{E}
\right)
\\
& \begin{aligned}
& \lesssim
\sum_{E \in \Omega_h} 
\Vert \EEi\Vert_{E} 
\Vert \TT \Vert_{W^{1,\infty}(E)}
( h_E \Vert\bnabla_h \eeh \Vert_{E} + \Vert \eeh \Vert_{E})
& & \text{(Lemma \ref{lm:bramble})}
\\
& \lesssim
\Vert \TT \Vert_{W^{1,\infty}(\Omega_h)}
\Vert \EEi\Vert 
\Vert \eeh\Vert 
\lesssim
\pss^{-1/2} \,
\Vert \TT \Vert_{W^{1,\infty}(\Omega_h)}
\Vert \EEi \Vert 
\normas{\eeh}
& & \text{(Lemma \ref{lm:inverse})}
\\
& \lesssim
\gs
h^{k} \vert \BB \vert_{k+1,\Omega_h} 
\normas{\eeh}
& & \text{(\eqref{eq:err-int} \& \eqref{eq:gamma})}
\end{aligned}
\end{aligned}
\end{equation}
For the estimate of the term $\alpha_{4,2}$ we proceed as follows. 
Being $\eeh \in \ZZd$ and $\Pzerok{0}\TT$ constant on each element, the vector calculus identity 
$$
\bcurl(A \times B) = (\diver B) \, A - (\diver A) \, B + (\bnabla A) B - (\bnabla B) A
$$
yields
\[
\bcurl_h(\eeh \times (\Pzerok{0}\TT)) = (\bnabla_h \eeh) (\Pzerok{0}\TT) := \pp_{k-1}  \in [\Pk_{k-1}(\Omega_h)]^3 \, .
\]
Therefore from \eqref{eq:int-orth}, the Cauchy-Schwarz inequality,
\eqref{eq:err-int}  and Lemma \ref{lm:cip}  
we infer
\[
\begin{aligned}
\alpha_{4,2} & = 
( \EEi \,, \pp_{k-1}  )
=
( \EEi \,, (I - \intcip)\pp_{k-1}) 
\\
&\leq
\biggl( \sum_{E \in \Omega_h} h_E^{-1}\Vert \EEi \Vert_{E}^2 \biggl)^{1/2} 
\biggl( \sum_{E \in \Omega_h} h_E\Vert (I - \intcip)\pp_{k-1}  \Vert_{E}^2 \biggl)^{1/2} 
\\
&\lesssim
h^{1/2} h^k \vert \BB \vert_{k+1,\Omega_h} 
\biggl(  \sum_{f \in  \EdgesI} h_f^2
\Vert \jump{\pp_{k-1}}_f \Vert_{f}^2 \biggr)^{1/2} \,.
\end{aligned}
\]
Furthermore
\[
\begin{aligned}
\sum_{f \in  \EdgesI} & h_f^2 
\Vert \jump{\pp_{k-1}}_f \Vert_{f}^2
\lesssim
\sum_{f \in \EdgesI} \left(
h_f^2  \Vert \jump{\bnabla_h \eeh}_f \TT \Vert_{f}^2 
+  
h_f^2
\Vert \jump{(\bnabla_h \eeh) ((I - \Pzerok{0})\TT)}_f\Vert_{f}^2 
\right)
\\
& \begin{aligned}
& \lesssim
\bdmJJ^{-1} \normacip{\eeh}^2 + 
\sum_{E \in \Omega_h} 
h_E \Vert (I - \Pzerok{0}) \TT \Vert_{L^\infty(E)}^2  
\Vert \bnabla_h \eeh \Vert_E^2 
& \quad & \text{(bound \eqref{eq:utile-jump})}
\\
& \lesssim
\bdmJJ^{-1} \normacip{\eeh}^2 + 
\sum_{E \in \Omega_h} 
h_E^3 \Vert \TT \Vert_{W^1_\infty(E)}^2  
\Vert \bnabla_h \eeh \Vert_E^2 
& \quad & \text{(Lemma \ref{lm:bramble})}
\\
& \lesssim
\bdmJJ^{-1} \normacip{\eeh}^2 + 
\pss^{-1}h
\Vert\TT\Vert_{W^{1,\infty}(\Omega_h)}^2  \normas{\eeh}^2
& \quad & 
\text{(Lemma \ref{lm:inverse})
}
\end{aligned}
\end{aligned}
\]
Hence, recalling definition \eqref{eq:gamma}, we have
\begin{equation}
\label{eq:alpha4-2}
\begin{aligned}
\alpha_{4,2} 
\lesssim
\gs \, h^{k}\vert \BB \vert_{k+1,\Omega_h} 
\normastab{\eeh}
  \,.
\end{aligned}
\end{equation}
Whereas for the term $\alpha_{4,3}$ we have
\begin{equation}
\label{eq:alpha4-3}
\begin{aligned}
\alpha_{4,3} & \leq
\biggl(\bdmJ^{-1} \sum_{f \in \EdgesI}
\Vert \EEi \Vert_{f}^2 \biggr)^{1/2} \normacip{\eeh}
& & \text{(Cauchy-Schwarz)}
\\
& \lesssim
\biggl(\bdmJ^{-1} \sum_{E \in \Omega_h}
(h_E^{-1} \Vert \EEi \Vert_{E}^2 + 
h_E \Vert \bnabla \EEi \Vert_{E}^2) \biggr)^{1/2} \normacip{\eeh}
& & \text{(bound \eqref{eq:utile-jump1})}
\\
& \lesssim
h^{1/2} h^k |\BB|_{k+1, \Omega_h} \normacip{\eeh}
\lesssim
\gs
h^{k} \vert \BB \vert_{k+1,\Omega_h} 
\normacip{\eeh}
& & \text{(\eqref{eq:err-int} \&  \eqref{eq:gamma})}
\end{aligned}
\end{equation}
Collecting \eqref{eq:alpha-4}, \eqref{eq:alpha4-1}, \eqref{eq:alpha4-2} and \eqref{eq:alpha4-3}, we finally obtain
\begin{equation}
\label{eq:alpha4-f}
\begin{aligned}
\alpha_4 
\lesssim
\gs \, h^{k}\vert \BB \vert_{k+1,\Omega_h} 
\normastab{\eeh}
  \,.
\end{aligned}
\end{equation}

\noindent
$\bullet$ Estimate of $\alpha_5$: we make the following computations

\begin{equation}
\label{eq:alpha5-f}
\begin{aligned}
\alpha_5 & = 
(\bcurl (\EEh) \times \TT \,, \eei )
= (\bcurl (\EEh) \times ((I - \Pzerok{0})\TT) \,, \eei  )
& & \text{(by \eqref{eq:orth-v})}
\\
& \leq \sum_{E \in \Omega_h}
\Vert (I - \Pzerok{0})\TT \Vert_{L^\infty(E)}
\Vert \bcurl (\EEh) \Vert_{E} 
\Vert \eei \Vert_{E}
& & \text{(Cauchy-Schwarz)}
\\
& \lesssim \sum_{E \in \Omega_h}
h_E \Vert \TT \Vert_{W^{1,\infty}(E)}
\Vert \bnabla \EEh \Vert_{E} 
\Vert \eei \Vert_{E}
& & \text{(Lemma \ref{lm:bramble})}
\\
& \lesssim 
\Vert \TT \Vert_{W^{1,\infty}(\Omega_h)}
\Vert \eei \Vert
\biggl(\sum_{E \in \Omega_h}
h_E^2 
\Vert \bnabla \EEh \Vert_{E}^2 \biggr)^{1/2} 
& & \text{(Cauchy-Schwarz)}
\\
& \lesssim 
\min \{ \psm^{-1/2},  \nm^{-1/2} h\} 
\Vert \TT \Vert_{W^{1,\infty}(\Omega_h)}
h^{k+1} \vert \uu \vert_{k+1, \Omega_h}
\normam{\EEh} 
& & \text{(Lemma \ref{lm:inverse} \& \eqref{eq:int-v})}
\\
& \lesssim 
\gm h^{k} \vert \uu \vert_{k+1, \Omega_h}
\normam{\EEh} 
& & \text{(def. \eqref{eq:gamma})}
\end{aligned}
\end{equation}

\noindent
$\bullet$ Estimate of $\alpha_6$: from the Cauchy-Schwarz inequality and \eqref{eq:eicip} we infer

\begin{equation}
\label{eq:alpha6-f}
\begin{aligned}
\alpha_6 &=
\sum_{f \in \EdgesI} \left( \bdmJ 
( \TT \times \jump{\eei}_f , \TT \times \jump{\eeh}_f )_f +
\bdmJJ h_f^2 
(\jump{\bnabla_h \eei}_f \TT ,\jump{\bnabla_h \eeh}_f \TT)_f
\right)
\\
& \leq
\normacip{\eei}
\normacip{\eeh}
\lesssim \ls \, h^k \vert \uu \vert_{k+1, \Omega_h} \normacip{\eeh} 
\lesssim \ls \, h^k \vert \uu \vert_{k+1, \Omega_h} \normastab{\eeh}
\,.
\end{aligned}
\end{equation}
Therefore collecting in \eqref{eq:conv0} the estimate \eqref{eq:alpha1-f}, \eqref{eq:alpha2-f}, \eqref{eq:alpha3-f}, \eqref{eq:alpha4-f}, \eqref{eq:alpha5-f}, \eqref{eq:alpha6-f}, we obtain
\[
\begin{aligned}
\normastab{\eeh}^2 + \normam{\EEh}^2
&\lesssim
(\ls + \gs) h^k \vert \uu\vert_{k+1, \Omega_h}  \normastab{\eeh}
+
\gs h^k \vert \BB\vert_{k+1, \Omega_h}  \normastab{\eeh}
+
\\
&+
(\lm \vert \BB\vert_{k+1, \Omega_h}  + \gm \vert \uu\vert_{k+1, \Omega_h}) h^k  \normam{\EEh} \,.
\end{aligned}
\]
The thesis easily follows from the Cauchy-Schwarz inequality.
\end{proof}

From Proposition \ref{prp:interpolation} and Proposition \ref{prp:error equation}
we immediately derive the following convergence result.
 
\begin{theorem}
\label{thm:conv}
Let Assumption \cfan{(MA1)} hold. Furthermore, if $k=1$ let also Assumption \cfan{(MA2)} hold.
Then, under the regularity assumption \cfan{(RA2)} and assuming that the parameter $\bdma$ (cf. \eqref{eq:forme_d}) is sufficiently large, referring to \eqref{eq:quantities} and \eqref{eq:gamma}, the following holds
\label{eq:full error equation}
\[
\begin{aligned}
\normastab{\uu - \uu_h}^2 &+ \normam{\BB - \BB_h}^2 \lesssim
\\
&
(\ls^2 + \gs^2 + \gm^2) h^{2k} \vert \uu \vert_{k+1, \Omega_h}^2 + (\lm^2 + \gs^2) h^{2k} \vert \BB \vert_{k+1, \Omega_h}^2 \,.
\end{aligned}
\]
\end{theorem}

\begin{remark}[Quasi-robustness]
The proposed method, assuming (as usual in the stationary case) the presence of a positive reaction term, guarantees $O(h^k)$ error estimates which are robust in the diffusion parameters. Furthermore, in a convection-dominated pre-asymptotic regime (that is when the quantities $\ns,\nm$ are dominated by the scaled convections $h \cc, h \TT$) an enhanced error reduction rate $O(h^{k+1/2})$ is obtained.
\end{remark}

\begin{remark}[Pressure robustness] 
\label{rem:press-rob}
As already observed, the proposed scheme is pressure robust in the sense of \cite{john-linke-merdon-neilan-rebholz:2017} (a modification of the continuous problem that only affects the pressure leads to changes in the discrete solution that only affect the discrete pressure). The estimate in Theorem \ref{thm:conv} reflects such property of the scheme, since the velocity error estimate is independent of the continuous pressure $p$. 
\end{remark}

\begin{theorem}[Error estimates for the pressure]
\label{thm:conv pre}
Under the same assumptions as in Theorem \ref{thm:conv} and
referring to \eqref{eq:quantities}, \eqref{eq:psi} and \eqref{eq:gamma}, the following holds
\label{eq:full error pre}
\[
\begin{aligned}
\Vert p - p_h \Vert
&\lesssim 
(\ts + \ls + \Vert \cc \Vert_{L^\infty(\Omega)} \pss^{-1/2} )
(1 + \bdma^{-1/2}) \normastab{\uu - \uu_h} +
\\ 
&  +
\Vert \TT \Vert_{W^1_\infty(\Omega_h)} \psm^{-1/2} (1 + \bdma^{-1/2})\normam{\BB - \BB_h} 
+ 
h^k \, \vert p \vert_{k, \Omega_h} \,.
\end{aligned}
\]
\end{theorem}

\begin{proof}
We simply sketch the proof since it is derived using standard arguments.
We preliminary observe that, combining the classical inf-sup arguments for the BDM element in \cite{boffi-brezzi-fortin:book} with the Poincar\'e inequality for piecewise regular functions \cite[Corollary 5.4]{DiPietro}, there exists $\ww_h \in \VVd$ such that
\begin{equation}
\label{eq:p-infsup}
\Vert \ww_h \Vert_{1,h} \lesssim 1 \,, 
\qquad \text{and} \qquad
\Vert p_h - \Pi_{k-1} p\Vert 
\lesssim b(\ww_h \,, p_h - \Pi_{k-1} p) \,.
\end{equation}
Furthermore, recalling the definition of $L^2$-projection operator, being $\diver(\VVd) \subseteq \Pk_{k-1}(\Omega_h)$, the following holds
\begin{equation}
\label{eq:p-orth}
b(\vv_h \,, p) =
b(\vv_h \,, \Pi_{k-1} p)
\qquad 
\text{for all $\vv_h \in \VVd$.}
\end{equation}
Therefore, combining \eqref{eq:p-infsup} and \eqref{eq:p-orth} with Problems \eqref{eq:linear variazionale} and \eqref{eq:linear fem}, and recalling \eqref{eq:forme_con},  we infer
\begin{equation}
\label{eq:p-error}
\begin{aligned}
\Vert p_h -  \Pi_{k-1} p\Vert  &
\lesssim
b(\ww_h \,, p_h - \Pi_{k-1} p)
=
b(\ww_h \,, p_h - p)
\\
& =
\bigl( \pss(\uu - \uu_h, \ww_h) + \ns \ash(\uu - \uu_h, \ww_h) \bigr)+
c_h(\uu - \uu_h, \ww_h) + \\  
&
-d(\BB - \BB_h, \ww_h)    + J_h(\uu - \uu_h, \ww_h)
 =: \sum_{i=1}^4 \beta_i \,.
\end{aligned}
\end{equation}
Using similar arguments to that in the proof of Proposition \ref{prp:error equation} we bound each term as follows:
\begin{equation}
\label{eq:beta}
\begin{aligned}
\beta_1 & \lesssim \ts (1 + \bdma^{-1/2}) \normastab{\uu - \uu_h} \,,
\\
\beta_2 & \lesssim (\ls \bdma^{-1/2} + 
\Vert \cc \Vert_{L^\infty(\Omega)} \pss^{-1/2}(1 + \bdma^{-1/2})) \normastab{\uu - \uu_h} \,,
\\
\beta_3 & \lesssim 
\Vert \TT \Vert_{W^1_\infty(\Omega_h)} \psm^{-1/2}(1 + \bdma^{-1/2}) \normam{\BB - \BB_h} \,,
\\
\beta_4 & \lesssim 
\ls(1 + \bdma^{-1/2}) \normastab{\uu - \uu_h} \,.
\end{aligned}
\end{equation}
The proof follows combining \eqref{eq:p-error} and \eqref{eq:beta} and Lemma \ref{lm:bramble}.
\end{proof}


\subsection{Convergence without requirements on solution regularity}
For completeness, in this section we include a convergence result for vanishing discretization parameter $h$, without any additional requirement on the solution regularity. In the following the  symbol $\lesssim$ will denote a bound up to a generic positive constant, independent of the mesh size $h$, but which may depend on the diffusive coefficients $\ns$ and $\nm$ and on the reaction coefficients $\pss$ and $\psm$,

\begin{proposition}
Let Assumption \cfan{(MA1)} hold. Let $\{ \uu_h, p_h, \BB_h \}_h$ denote a sequence of solutions of the discrete problem \eqref{eq:linear fem} with mesh parameter $h$ tending to zero, and let $(\uu,p,\BB) \in \VVc \times \Qc \times \WWc$ be solution of the continuous problem \eqref{eq:linear variazionale}. 
Then, for vanishing $h$, it holds (for any $2 \le p < 6$)
$$
\begin{aligned}
& \uu_h \rightarrow \uu \quad &\text{in $L^p(\Omega)$} \, , 
\qquad &p_h \rightharpoonup p \quad &\text{weakly in  $L^2(\Omega)$} \, , \\
& \BB_h \rightarrow \BB \quad &\text{in $L^p(\Omega)$} \,, 
\qquad &\BB_h \rightharpoonup \BB \quad &\text{weakly in  $H^1(\Omega)$} \, .
\end{aligned}
$$
\end{proposition}
\begin{proof}
We present the proof briefly since it follows quite standard arguments.
We split the proof into several steps.

\noindent 
\textsl{Step 1: A priori estimate and convergence for $\BB_h$ and $p_h$.}
We start by observing that, first combining \eqref{eq:linear stab} with Proposition \ref{prp:coe}, then using classical inf-sup arguments for the BDM element, one easily obtains
\begin{equation}
\label{eq:apriori}
\normastab{\uu_h}^2 + \normam{\BB_h}^2 + \| p_h \|^2 \leq C 
\end{equation}
where the constant $C$ depends on $\ff,\GG$ but is uniform in $h$.
As a consequence, there exist $\BB \in [H^1(\Omega)]^2$, $p \in L^2(\Omega)$ and suitable subsequences (which we still denote by $\BB_h$ and $p_h$ for simplicity) such that
\begin{equation}
\label{eq:later0}
\begin{aligned}
\BB_h \rightharpoonup \BB \quad \text{weakly in  $H^1(\Omega)$,}
\quad  
\BB_h \rightarrow \BB \quad \text{in  $L^p(\Omega)$,} 
\quad p_h \rightharpoonup p \quad \text{weakly in  $L^2(\Omega)$.}
\end{aligned}
\end{equation}
Note that, here above, we commit a small abuse by denoting already such limits with $\BB$ and $p$ (we will actually prove below in \textsl{Step 4} that such limits are the solutions of the continuous problem).

\noindent
\textsl{Step 2: Convergence for $\uu_h$.}
The velocity variable is more involved, due to the non-conformity of the discrete space $\VVd$ (cf. \eqref{eq:spazi_d}). 
We start by introducing $\uuht = \intcip (\uu_h)$, where $\intcip \colon \Pk_{k}(\Omega_h) \to \Pkc_{k}(\Omega_h)$ denotes the Oswald interpolation  operator \cite{Hoppe} (cf. also Lemma \ref{lm:cip}). 
For any $E \in \Omega_h$, let $\Edgesloc$  denote the set of faces in $\Edges$ with closure that has non-null intersection with the closure of $E$.
Then employing Lemma 3.2 in \cite{CIP} and \textbf{(MP1)}, for any $E \in \Omega_h$ the following holds
\begin{equation}\label{later:1}
\| \uu_h - \uuht \|_{E}^2 \lesssim
\sum_{f \in \Edgesloc}
h_f \Vert \jump{\uu_h}_f \Vert_f^2 
\lesssim h_E^2 \sum_{f \in \Edgesloc}
h_f^{-1} \Vert \jump{\uu_h}_f \Vert_f^2 \, .
\end{equation}
First by the Korn inequality, by simple manipulations,  by inverse estimate (cf. Lemma \ref{lm:inverse}) coupled with \eqref{later:1}, and finally employing the a priori estimate \eqref{eq:apriori}, we infer
$$
\begin{aligned}
\Vert \uuht \Vert^2 +  
\Vert \bnabla \uuht \Vert^2 &\lesssim 
\Vert \uuht \Vert^2 + \| \beps(\uuht) \|_\Omega^2 
\\
&\lesssim 
\sum_{E \in \Omega_h} \Big( \Vert \uu_h - \uuht \Vert_{E}^2 + 
\Vert \bnabla(\uu_h - \uuht) \Vert_{E}^2 \Big) +
\Vert \uu_h \Vert^2_E + 
\Vert \beps(\uu_h) \Vert^2
\\
& \lesssim \sum_{E \in \Omega_h}
\sum_{f \in \Edgesloc}
h_f^{-1} \Vert \jump{\uu_h}_f \Vert_f^2 +
\Vert \uu_h \Vert^2_E + 
\Vert \beps(\uu_h) \Vert^2
\\
& \lesssim \normastab{\uu_h}^2  \leq C \, .
\end{aligned}
$$
As a consequence, since $\uuht \in [H^1(\Omega)]^3$, by standard arguments it exists a sub-sequence such that (for $2 \leq p < 6$)
\begin{equation}\label{later:3}
\uuht \rightharpoonup \uu \quad \textrm{weakly in } H^1(\Omega) \, , 
\quad 
\uuht \rightarrow \uu \quad \textrm{in } L^p(\Omega) \, ,
\end{equation}
where again we label the limit already as $\uu$ with an abuse.
Let now, for any $2 \leq p < 6$, the positive real $\rho := 6/p -1$.
We now combine an inverse estimate with \eqref{later:1} and bound \eqref{eq:apriori}, obtaining for any $E \in \Omega_h$
\begin{equation}\label{later:2}
\begin{aligned}
 \| \uu_h - \uuht \|_{L^p(E)}^2 &\lesssim 
h_E^{6/p-3} \| \uu_h - \uuht \|_E^2 \lesssim h_E^\rho \sum_{f \in \Edgesloc}
h_f^{-1} \Vert \jump{\uu_h}_f \Vert_f^2\, .
\end{aligned}
\end{equation}
Notice that \eqref{later:2} easily implies
\begin{equation}
\label{later:2bis}
\max_{E \in \Omega_h} \| \uu_h - \uuht \|_{L^p(E)} 
\lesssim h^{\rho/2} \normastab{\uu_h}  \leq h^{\rho/2} C^{1/2}
\lesssim 1  \,.
\end{equation}
Then from \eqref{later:2} and \eqref{later:2bis}, we derive
$$
\begin{aligned}
\| \uu_h - \uuht \|_{L^p(\Omega)}^p &= 
\sum_{E \in \Omega_h} \| \uu_h - \uuht \|_{L^p(E)}^{p-2} \| \uu_h - \uuht \|_{L^p(E)}^2 
\lesssim \sum_{E \in \Omega_h} \| \uu_h - \uuht \|_{L^p(E)}^2 \\
& \lesssim \sum_{E \in \Omega_h} h_E^\rho \sum_{f \in \Edgesloc}
h_f^{-1} \Vert \jump{\uu_h}_f \Vert_f^2
\lesssim h^\rho \normastab{\uu_h}^2  \leq h^\rho C \, .
\end{aligned}
$$
The above bound, combined with \eqref{later:3} immediately yields (for $h \rightarrow 0$)
\begin{equation}\label{later:4}
\uu_h \rightarrow \uu \qquad \textrm{in } L^p(\Omega) \, .
\end{equation}

\noindent
\textsl{Step 3: The limit $(\uu, p, \BB) \in \VVc \times \Qc \times \WWc$.}
Employing analogous arguments to that in \textsl{Step 2}, we can obtain 
\[
\begin{aligned}
\Vert \uuht\Vert^2_{\partial \Omega}   
& \lesssim
\sum_{f \in \EdgesB} \Vert \uuht - \uu_h\Vert^2_{f} + 
\sum_{f \in \EdgesB} \Vert \uu_h\Vert^2_{f} 
\\
& \lesssim
\sum_{f \in \Edges} \Vert \uuht - \uu_h\Vert^2_{f} + 
h \sum_{f \in \Edges} h_f^{-1}\Vert \jump{\uu_h}_f\Vert^2_{f} 
\\
& \lesssim
\sum_{E \in \Omega_h} h_E^{-1} \Vert \uuht - \uu_h\Vert^2_{E} + 
h \normastab{\uu_h}^2
&\quad & \text{(by \eqref{eq:utile-jump})}
\\
& \lesssim
h \normastab{\uu_h}^2 \leq h C
&\quad & \text{(by \eqref{later:1} \& \eqref{eq:apriori})}
\end{aligned}
\]
that is $\uuht|_{\partial\Omega}$ converges to zero in $L^2(\partial \Omega)$ as $h \rightarrow 0$.
Combining the result above with compact Sobolev embeddings and trace theorem, we derive (here we can take any $1/2 <s < 1$)
\begin{equation}
\label{eq:step3}
\Vert \uu  \Vert_{\partial \Omega} = 
\lim_{h \to 0} \Vert \uu  - \uuht \Vert_{\partial \Omega} 
\lesssim
\lim_{h \to 0} \Vert \uu  - \uuht \Vert_{s, \Omega} = 0\,,
\end{equation}
i.e. $\uu \in \VVc$. 
The same argument in \eqref{eq:step3} easily applies to $\BB \cdot \nn$, therefore $\BB \in \WWc$ (an alternative way to verify the boundary conditions for $\BB$ is to exploit that convex subsets which are closed in the norm topology are closed also in the weak topology). We finally notice that from \eqref{eq:later0} we infer
$(p, 1) = \lim_{h \to 0 } (p_h, 1) = 0$ i.e. $p \in \Qc$.

\noindent
\textsl{Step 4: The limit $(\uu, p, \BB)$ solves \eqref{eq:linear variazionale}.}
We start by proving that $\uu$ realizes the second equation in \eqref{eq:linear variazionale}.
Since $\diver \, \uu_h=0$ for all $h$, from \eqref{later:4} we have 
\[
b(\uu, q)= - (\uu, \nabla q) = 
- \lim_{h \to 0} (\uu_h, \nabla q) =
\lim_{h \to 0} b(\uu_h, q) = 0 
\quad \text{for all $q \in C^\infty_0(\Omega)$.}
\]
The condition $b(\uu, q)=0$ for all $q \in \Qc$  now follows by density arguments.

Let us now analyse the first equation in \eqref{eq:linear variazionale}.
For any $h>0$, let $X \colon \VVc \to \R$ and $X_h \colon \VVd \to \R$ be the linear forms defined by
\[
\begin{aligned}
X(\ww) &:= \pss (\uu, \ww) + \ns \as(\uu, \ww) + c(\uu, \ww)  
-d(\BB, \ww)  + b(\ww, p) \,,
\\
X_h(\ww) &:= \pss (\uu_h, \ww) + \ns \ash(\uu_h, \ww) + c_h(\uu_h, \ww) -d(\BB_h, \ww)  + J_h(\uu_h, \ww) + b(\ww, p_h) \,.
\end{aligned}
\]
Notice that, from \eqref{later:4} and \eqref{eq:later0}, for any $\ww \in  [ C^\infty_0(\Omega)]^2$ we infer
\[
\begin{aligned}
\pss (\uu, \ww) &= \lim_{h \to 0} \pss (\uu_h, \ww) \,;
\\
\ns \as(\uu, \ww) &= - \ns(\uu, \diver (\beps(\ww))) 
= \lim_{h \to 0} - \ns(\uu_h, \diver (\beps(\ww)))
\\
&= \lim_{h \to 0}  \ns\biggl(
(\beps_h(\uu), \beps(\ww)) - \sum_{f \in \Sigma_h}(\jump{\uu_h}_f, \media{\beps(\ww)}_f)_f
\biggr) 
\\
&= \lim_{h \to 0} \ns \ash(\uu_h, \ww)  \,;
\\
c(\uu, \ww) &= - (\uu, (\bnabla \ww) \cc) = 
\lim_{h \to 0} - (\uu_h, (\bnabla \ww) \cc) 
\\
&= \lim_{h \to 0} \bigl((( \bnabla_h \uu_h ) \, \cc, \, \ww )
- \sum_{f \in \EdgesI} ( (\cc \cdot \nn_f) \jump{\uu_h}_f ,\, \media{\ww}_f)_f \bigr) 
\\
&
= \lim_{h \to 0} c_h(\uu_h, \ww) \,;
\\
-d(\BB, \ww) &= (\BB, \bcurl(\ww \times \TT)) = 
\lim_{h \to 0} (\BB_h, \bcurl(\ww \times \TT)) =
\lim_{h \to 0} -d(\BB_h, \ww) \,;
\\
b(\ww, p) &= \lim_{h \to 0} b(\ww, p_h) \,;
\end{aligned}
\]
therefore, being $J_h(\uu_h, \ww) = 0$, we conclude that
\begin{equation}
\label{eq:step4-2}
X(\ww) = \lim_{h \to 0} X_h(\ww) 
\quad \text{for all $\ww \in  [ C^\infty_0(\Omega)]^2$.}
\end{equation}
For any $h>0$, let $\PVVd^h \colon \VVc \to \VVd$ be the interpolation operator of Lemma \ref{lm:int-v}, then, direct computations, yield
\begin{equation}
\label{eq:step4-3}
\begin{aligned}
 X_h(\ww)  &=  X_h(\PVVd^h \ww)  +  X_h(\ww - \PVVd^h\ww) 
 = (\ff, \PVVd^h \ww) + X_h(\ww - \PVVd^h\ww) \\
 &= (\ff, \ww) + X_h(\ww - \PVVd^h\ww) + (\ff, \PVVd^h\ww - \ww) \quad \text{for all $\ww \in  [ C^\infty_0(\Omega)]^2$.}
\end{aligned}
\end{equation}
Using similar arguments to that in Proposition \ref{prp:interpolation} and employing \eqref{eq:apriori}, we  derive
\[
\begin{aligned}
|X_h(\ww - \PVVd^h\ww) + (\ff, \PVVd^h\ww - \ww)| &\lesssim
\bigl(\normastab{\uu_h} + \normam{\BB_h} + \Vert p_h \Vert + \Vert \ff \Vert) h^{k} |\ww|_{k+1, \Omega_h} 
\\
&\lesssim C h^{k} |\ww|_{k+1, \Omega_h} 
\,.
\end{aligned}
\]
Combining the previous bound with \eqref{eq:step4-2} and \eqref{eq:step4-3} we obtain
\[
\pss (\uu, \ww) + \ns \as(\uu, \ww) + c(\uu, \ww)  
-d(\BB, \ww)  + b(\ww, p) = 
(\ff, \ww)  
\quad  \text{$\forall \ww \in  [ C^\infty_0(\Omega)]^2$,} 
\]
and by density argument we recover the first equation in \eqref{eq:linear variazionale}.
The third equation in \eqref{eq:linear variazionale} can be derived using analogous techniques.

\noindent
\textsl{Step 5: Convergence of the whole sequence $\{\uu_h, p_h, \BB_h\}_h$.} We finally note that since the linear problem \eqref{eq:linear variazionale} admits a unique solution $(\uu, p, \BB)$, using standard arguments, we conclude that the whole sequence $\{\uu_h, p_h, \BB_h\}_h$ converges to $(\uu, p, \BB)$.
\end{proof}

\section{Numerical Experiments}\label{sec:num}

In this section we will present a set of numerical tests supporting the theoretical results derived in Section \ref{sec:theo}. 

To analyze the robustness of the proposed method under a convection-dominant regime, 
we will consider the same problem in each subsection. 
However, we will vary the values of $\ns$ and $\nm$ to explore different scenarios.
We solve the problem defined in Equation~\eqref{eq:linear variazionale} in $\Omega=(0,\,1)^3$
where the data are chosen in such a way that the exact solution is:
\[
\begin{gathered}
\uu(x,\,y,\,z):= \left[\begin{array}{r}
                  \sin(\pi x)\cos(\pi y)\cos(\pi z)\\
                  \sin(\pi y)\cos(\pi z)\cos(\pi x)\\
                  2\sin(\pi z)\cos(\pi x)\cos(\pi y)\\
                 \end{array}\right]\,,              
\qquad
\BB(x,\,y,\,z):= \begin{bmatrix}
                  \sin(\pi y)\\
                  \sin(\pi z)\\
                  \sin(\pi x)
                 \end{bmatrix}\,,
\\                                  
p(x,\,y,\,z) := \sin(x)+\sin(y)-2\sin(z)\,,
\end{gathered}
\]
and we set $\cc=\uu$ and $\TT=\BB$, in accordance with the underlying non-linear MHD problem. Then, the reaction coefficients $\pss$ and $\psm$ are fixed to 1.

Given a tetrahedral mesh of $\Omega$, denoted as usual by $\Omega_h$, 
we define its mesh size as the mean value of tetrahedrons diameters, i.e.,
$h := \min_{E\in\Omega_h} h_E\,.$
In order to develop the numerical convergence analysis, 
we construct a family of Delaunay tetrahedral meshes with decreasing mesh size
using the library \texttt{tetgen}~\cite{Si:2015:TAD}.
We refer to these meshes as \mesh{0},$\,\,$\mesh{1},$\,\,$\mesh{2},$\,\,$\mesh{3} and \mesh{4}.
As an example, in Figure~\ref{fig:mesh2} we show \mesh{2}.
In the subsequent numerical analysis we consider $k=1$ and $k=2$.
In order to partially balance the computational costs, will use the set of meshes \mesh{1},$\,\,$\mesh{2}, $ $\mesh{3} and \mesh{4} for $k=1$,
and the set \mesh{0},$\,\,$\mesh{1}, \mesh{2} and \mesh{3} for $k=2$.

\begin{figure}[!htb]
\centering 
\begin{tabular}{ccc}
\includegraphics[width=0.35\textwidth]{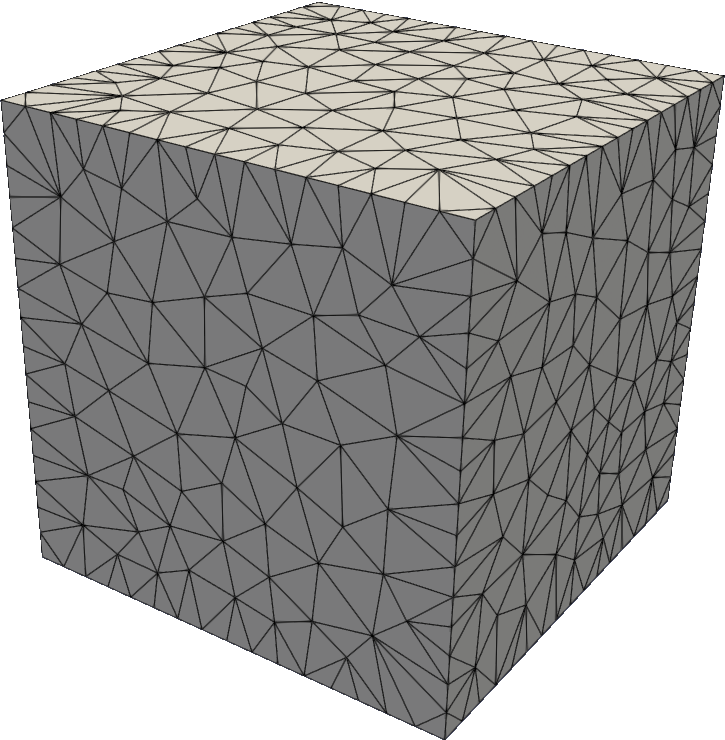}&
\qquad
\includegraphics[width=0.35\textwidth]{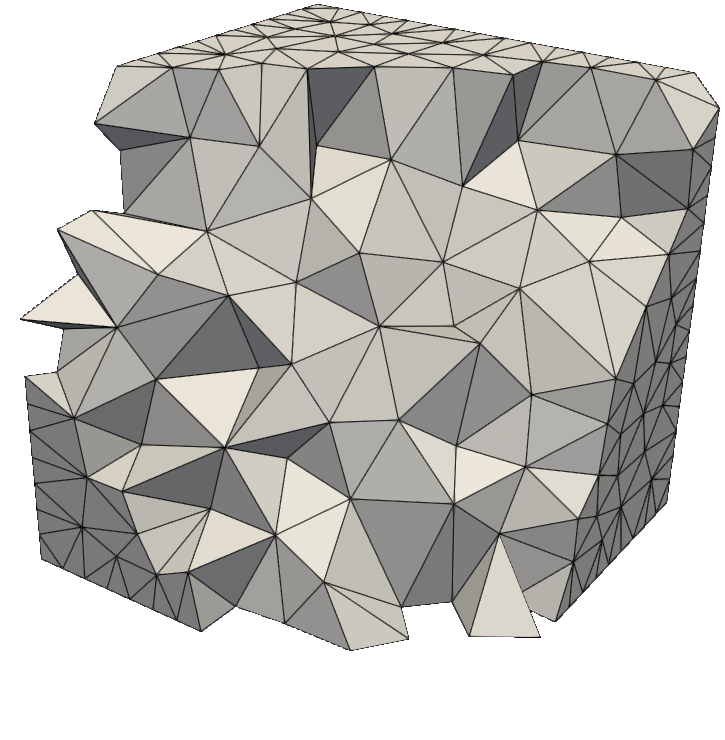}
\end{tabular}
\caption{\mesh{2} and a detail inside.}
\label{fig:mesh2}
\end{figure}

We use the label \mfs{} to refer to method \eqref{eq:linear fem}, that includes a stabilization term for both the fluid and magnetic fields.
For comparison purposes, we consider also a more standard upwind stabilization for the fluid equations, without the additional stabilizing form $J_h$ (cf. \eqref{eq:J}); we denote such scheme by \fs{}.
Then, following~\cite{upwinding, DiPietro},
we set $\bdmc=1$, $\bdmJ = 5$ and $\bdmJJ=0.01$ for both approximation degrees,
while $\bdma$ will be equal to 10 and 20 for $k=1$ and $k=2$, respectively.

\subsection{Fluid and magneto convective dominant regime}

In  this test we consider a convective dominant regime. 
We set $\ns=\nm=\nu$ whit $\nu = \texttt{1.0e-05},\texttt{1.0e-10}$.
We examine the behaviour of the proposed method looking at different error indicators.
More specifically, we use the $H^1$ semi-norm and $L^2$ norm errors for both velocity and magnetic fields,
the $L^2$ norm for the pressure and the norm $\normastab{\cdot}$ for the velocity field  (cf. ~\eqref{eq:normeVW}).

\begin{figure}[!htb]
\centering
\begin{tabular}{cc}
\includegraphics[width=0.46\textwidth]{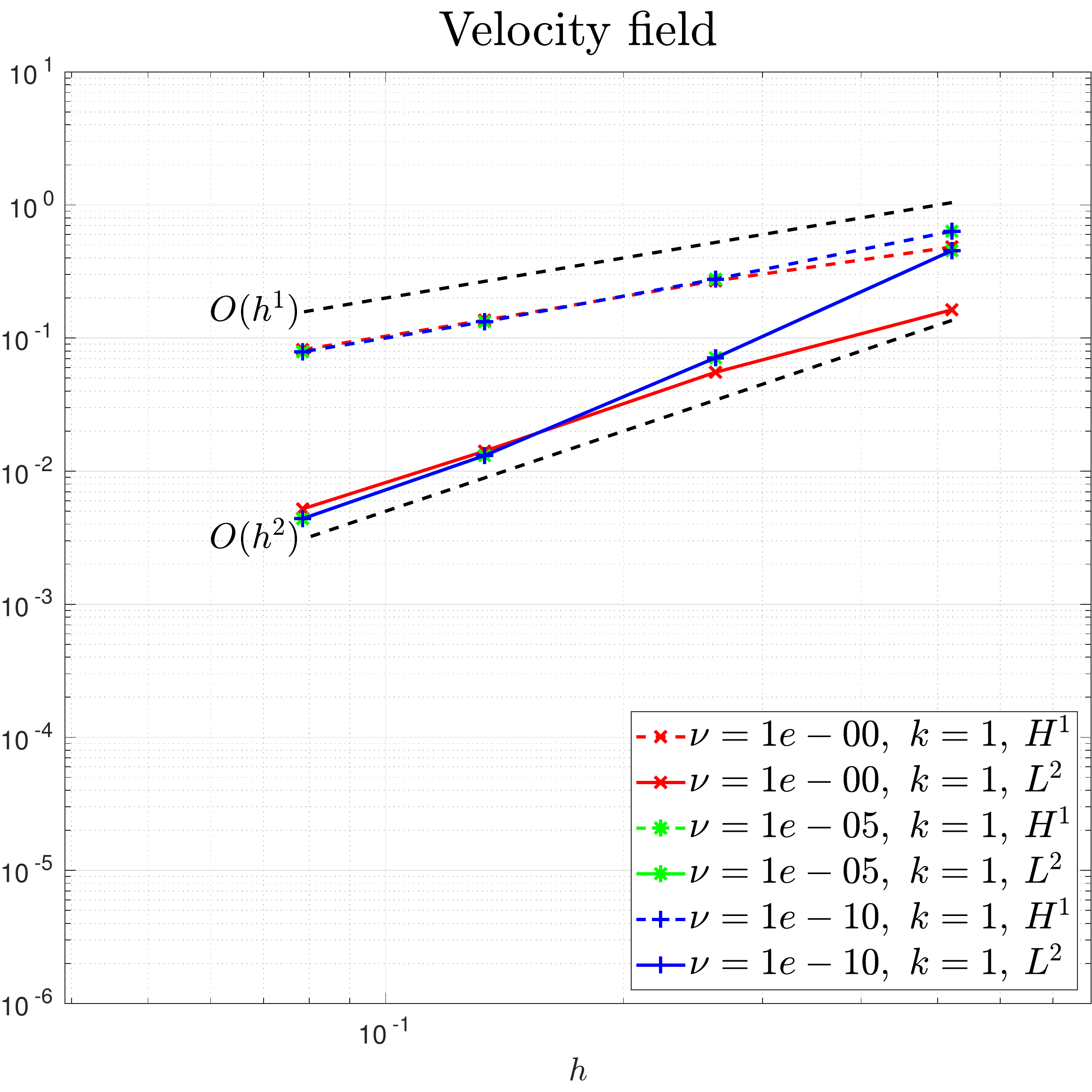}& 
\includegraphics[width=0.46\textwidth]{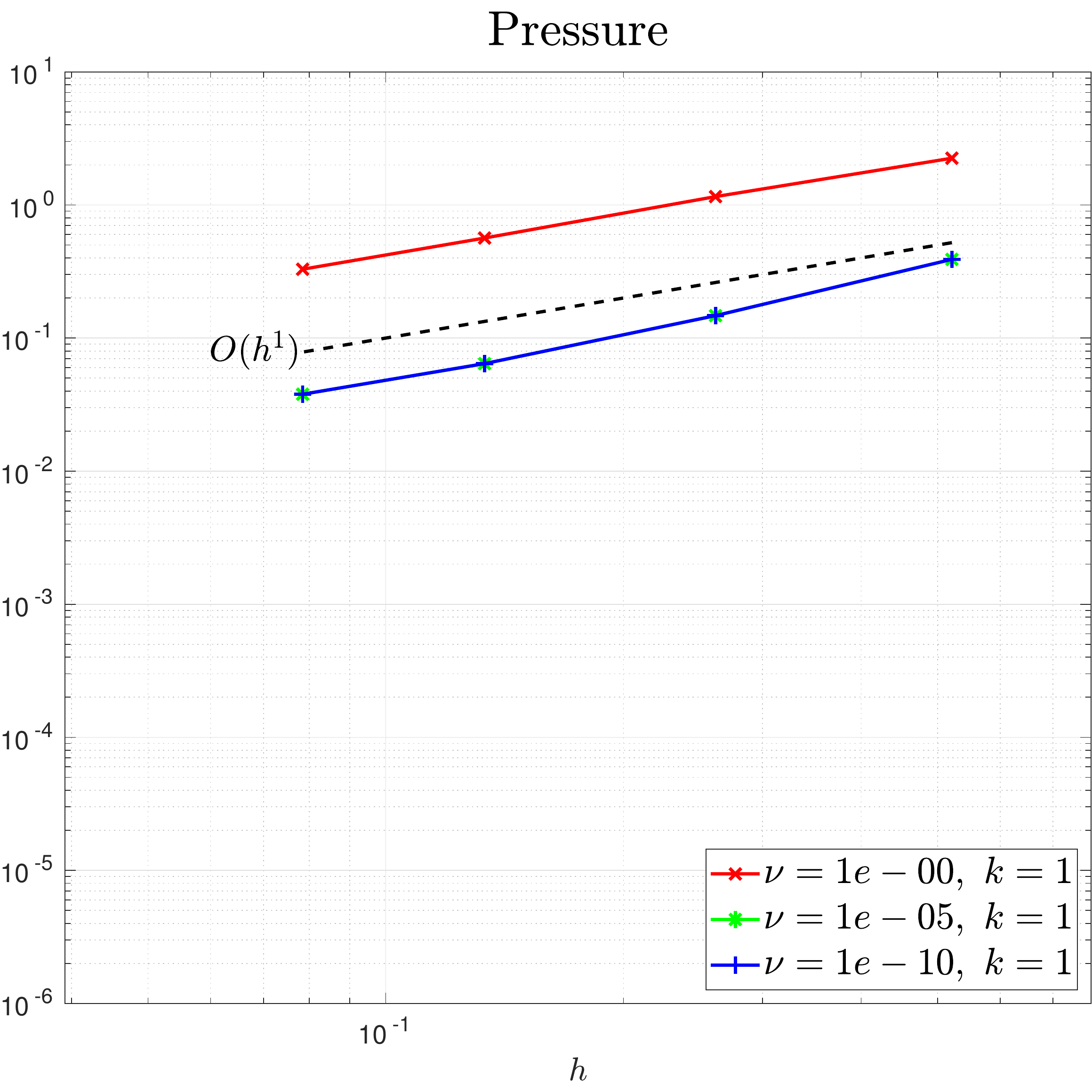}\\
\includegraphics[width=0.46\textwidth]{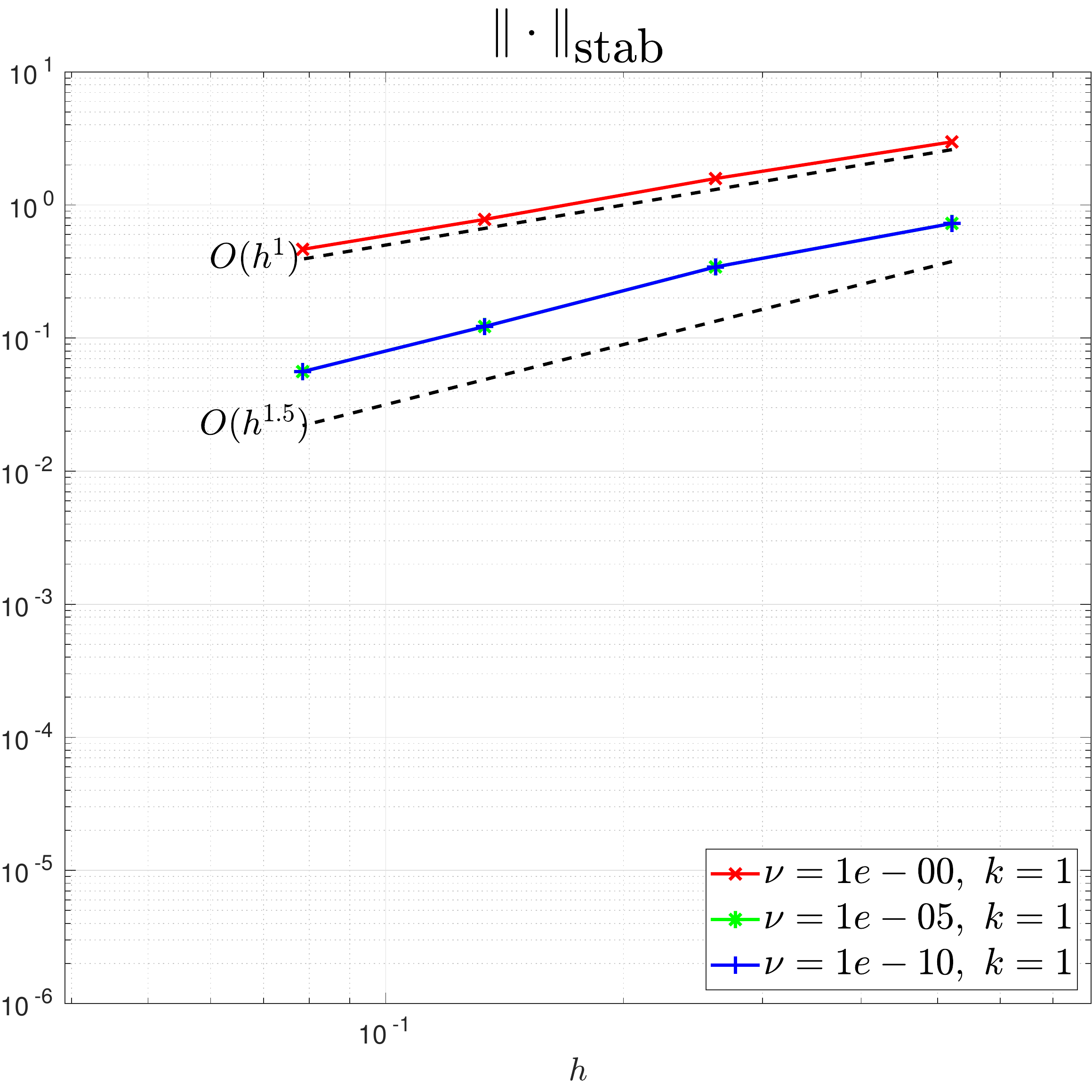}&
\includegraphics[width=0.46\textwidth]{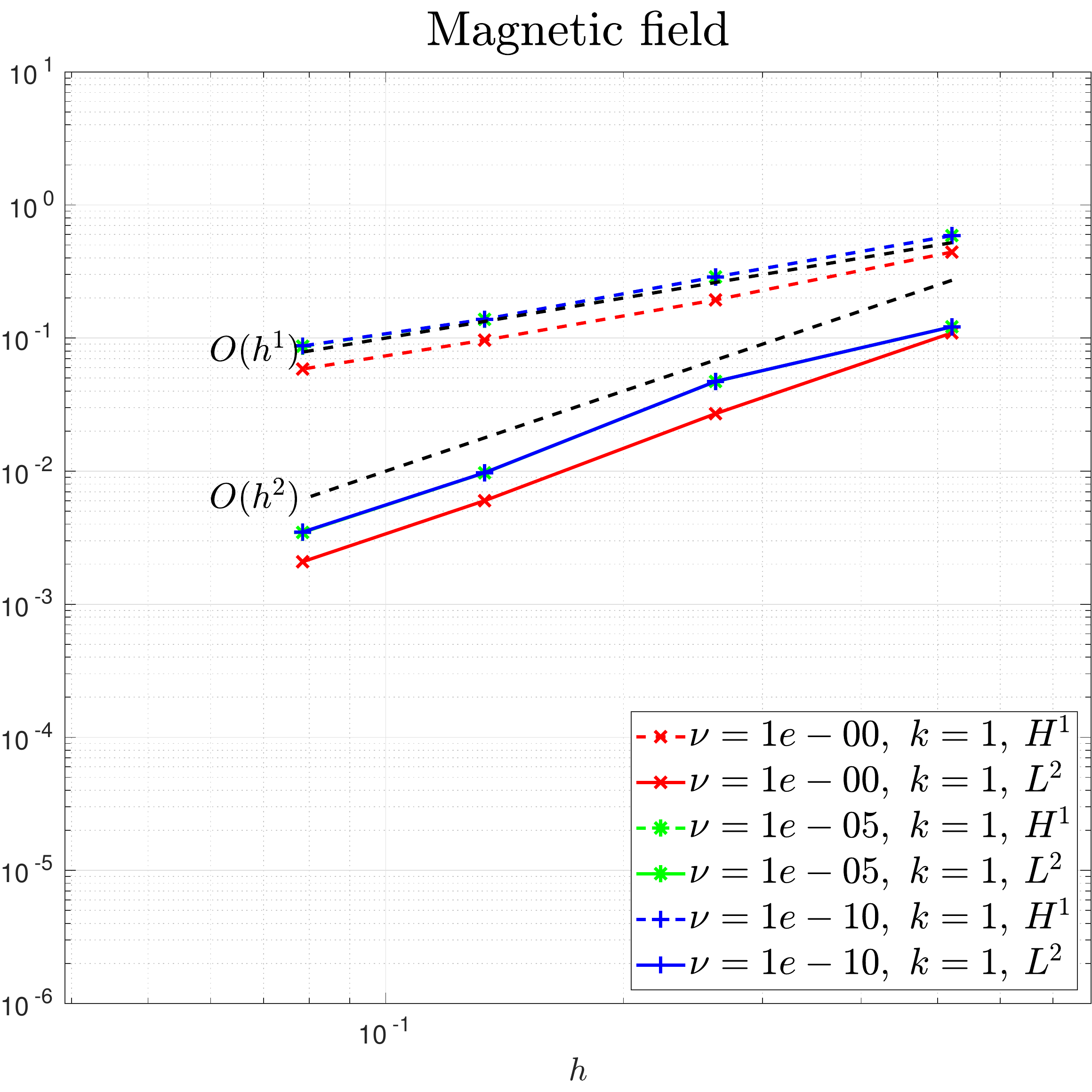}\\
\end{tabular}
\caption{Convection dominant regime with different values of $\ns$ and $\nm$, $k=1$.}
\label{fig:exe1_2_dgr_1}
\end{figure}

\begin{figure}[!htb]
\centering
\begin{tabular}{cc}
\includegraphics[width=0.46\textwidth]{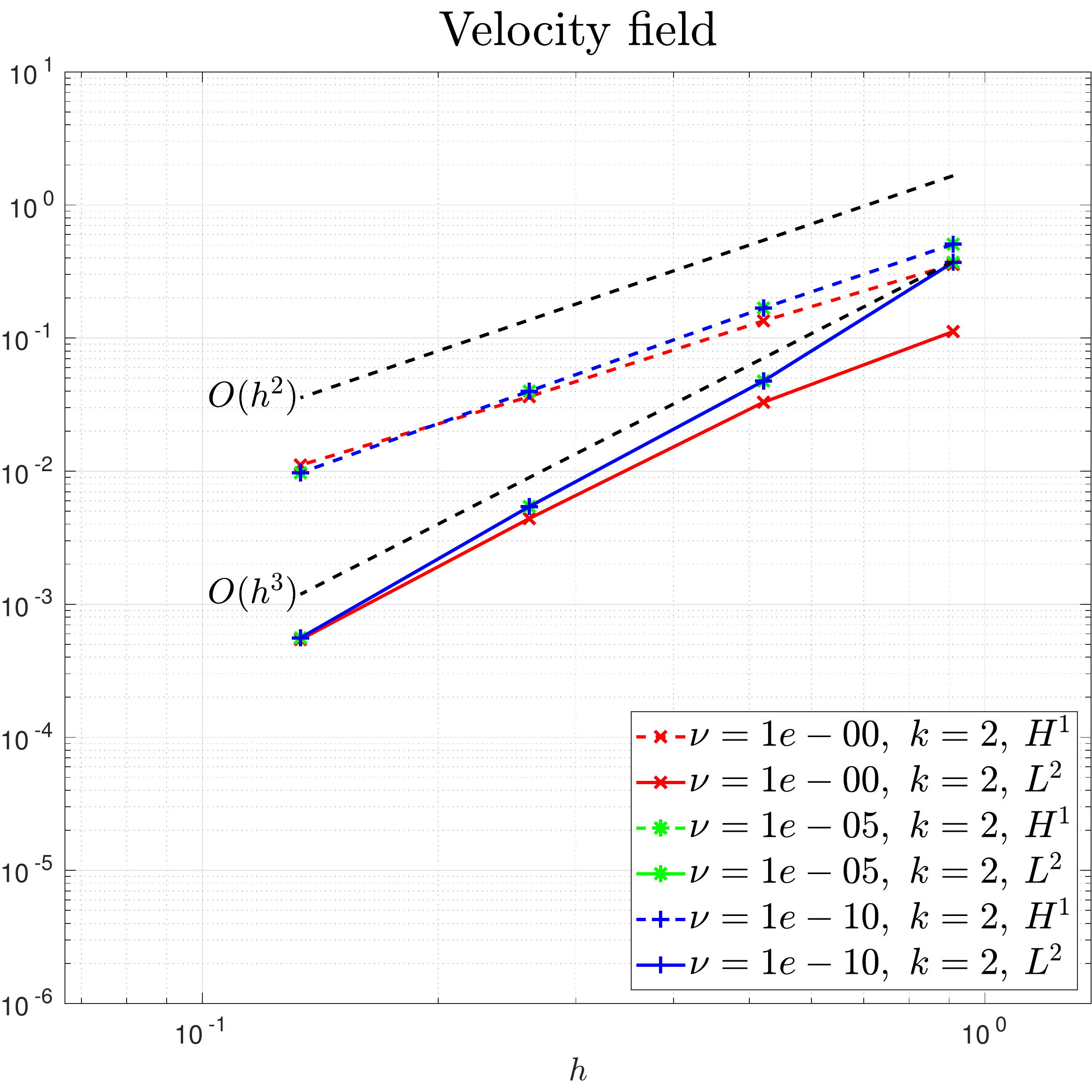}& 
\includegraphics[width=0.46\textwidth]{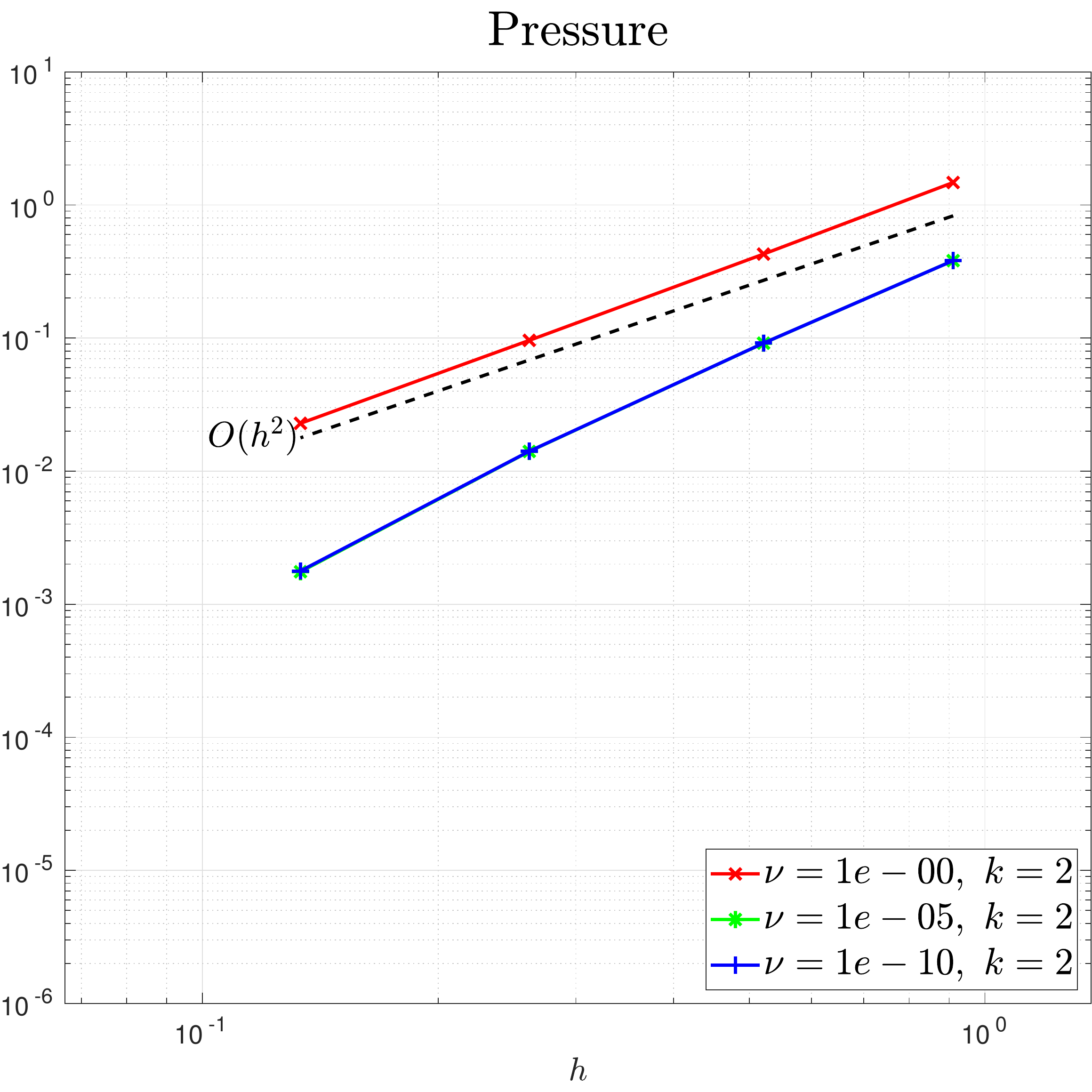}\\
\includegraphics[width=0.46\textwidth]{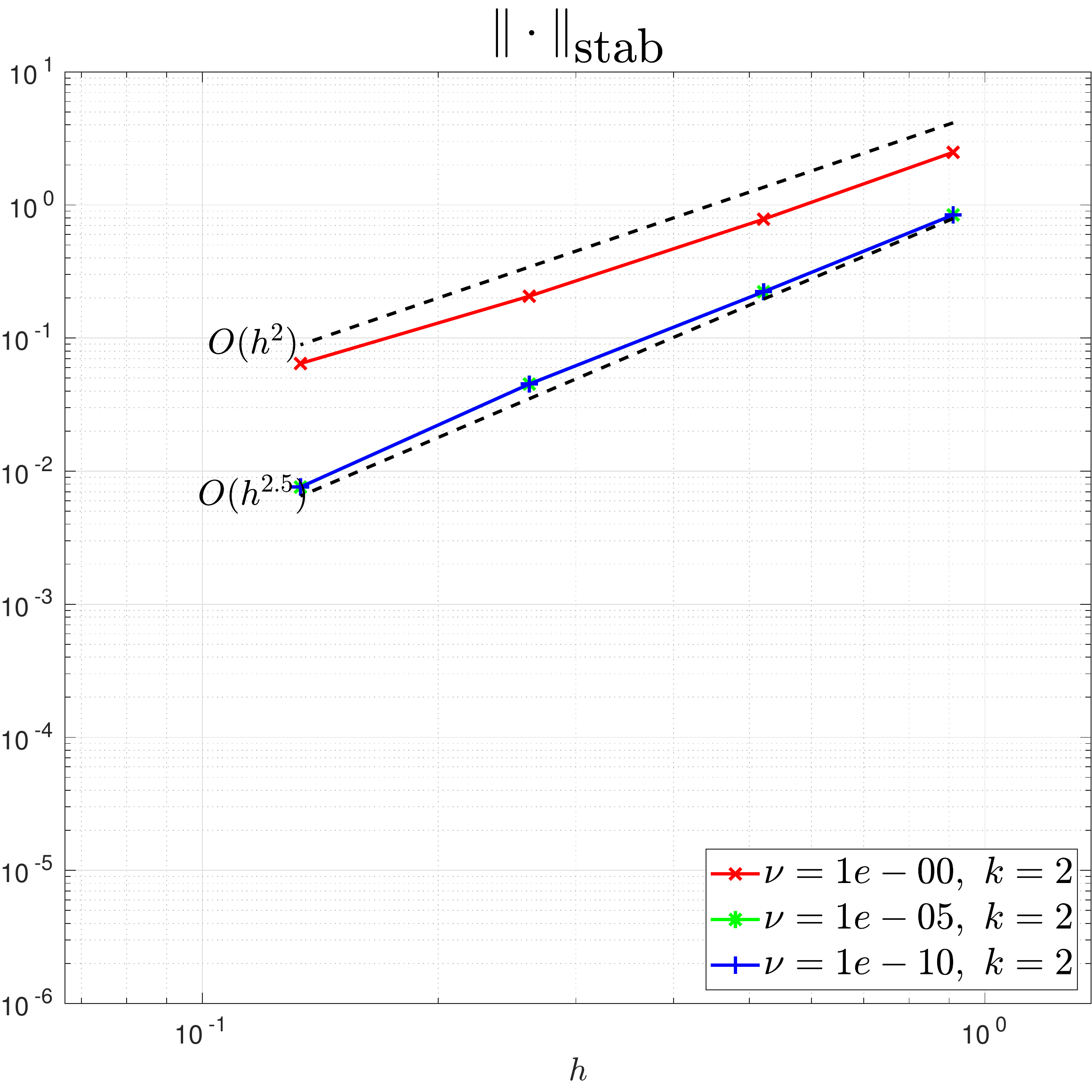}&
\includegraphics[width=0.46\textwidth]{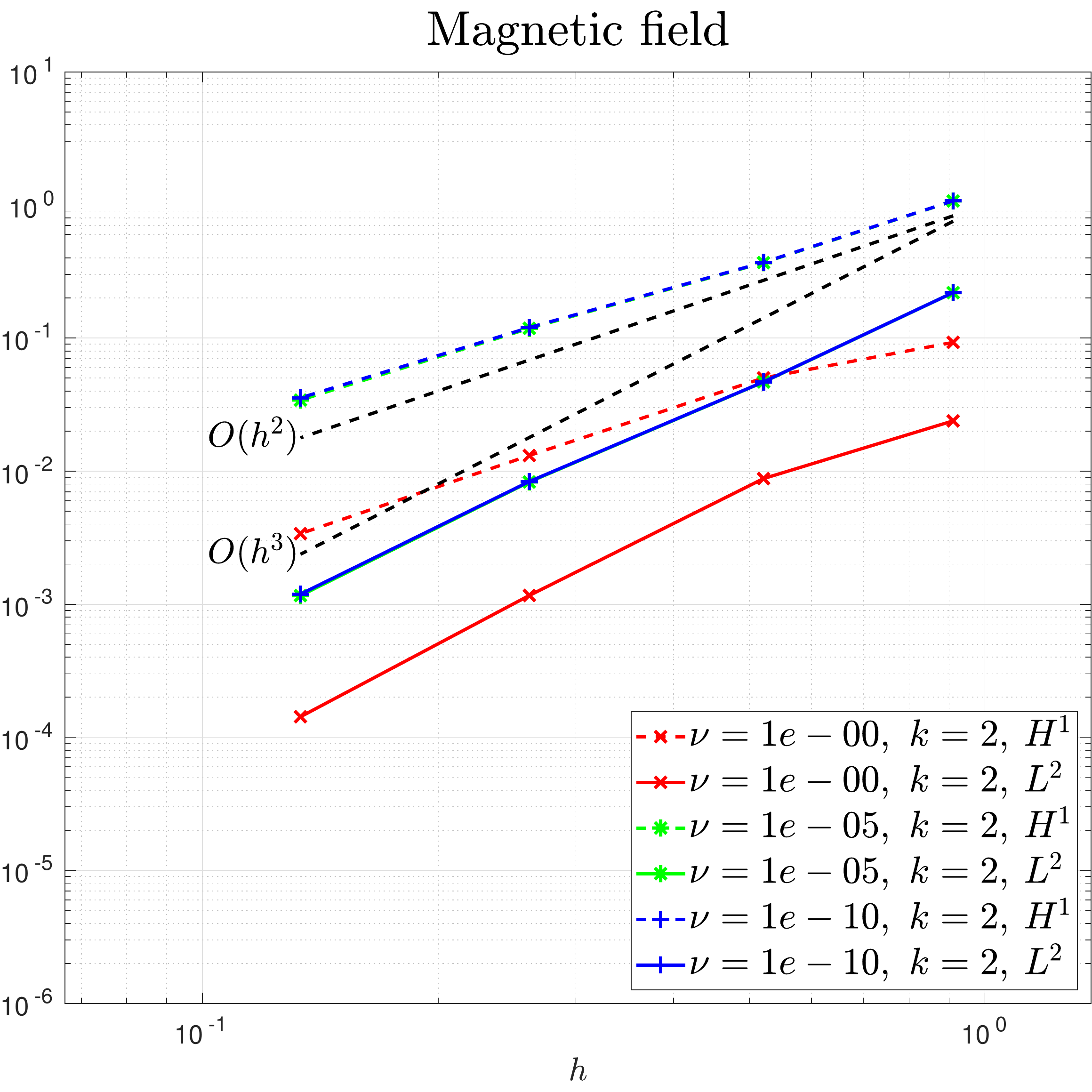}\\
\end{tabular}
\caption{ Convection dominant regime with different values of $\ns$ and $\nm$, $k=2$.}
\label{fig:exe1_2_dgr_2}
\end{figure}

In Figures~\ref{fig:exe1_2_dgr_1} and~\ref{fig:exe1_2_dgr_2}, 
we present the convergence trend of all these error indicators for each choice of the parameter $\nu$.
These graphs also include the trend in a diffusive dominant regime, $\nu = \texttt{1e-00}$,
for comparison with the convective dominant regime.
All the error indicators for all variables exhibit a behaviour consistent with the theoretical results for each degree $k$.

It is worth inspecting more in deep the trend of $\normastab{\cdot}$.
According to the theory, in a diffusive dominant regime the convergence rate has to be $O\left(h^k\right)$,
while in a (pre-asymptotic) convective dominant regime the error reduction rate is $O\left(h^{(2k+1)/2}\right)$:
the numerical data precisely align with this behaviour.

\begin{figure}[!htb]
\centering
\begin{tabular}{cc}
\includegraphics[width=0.46\textwidth]{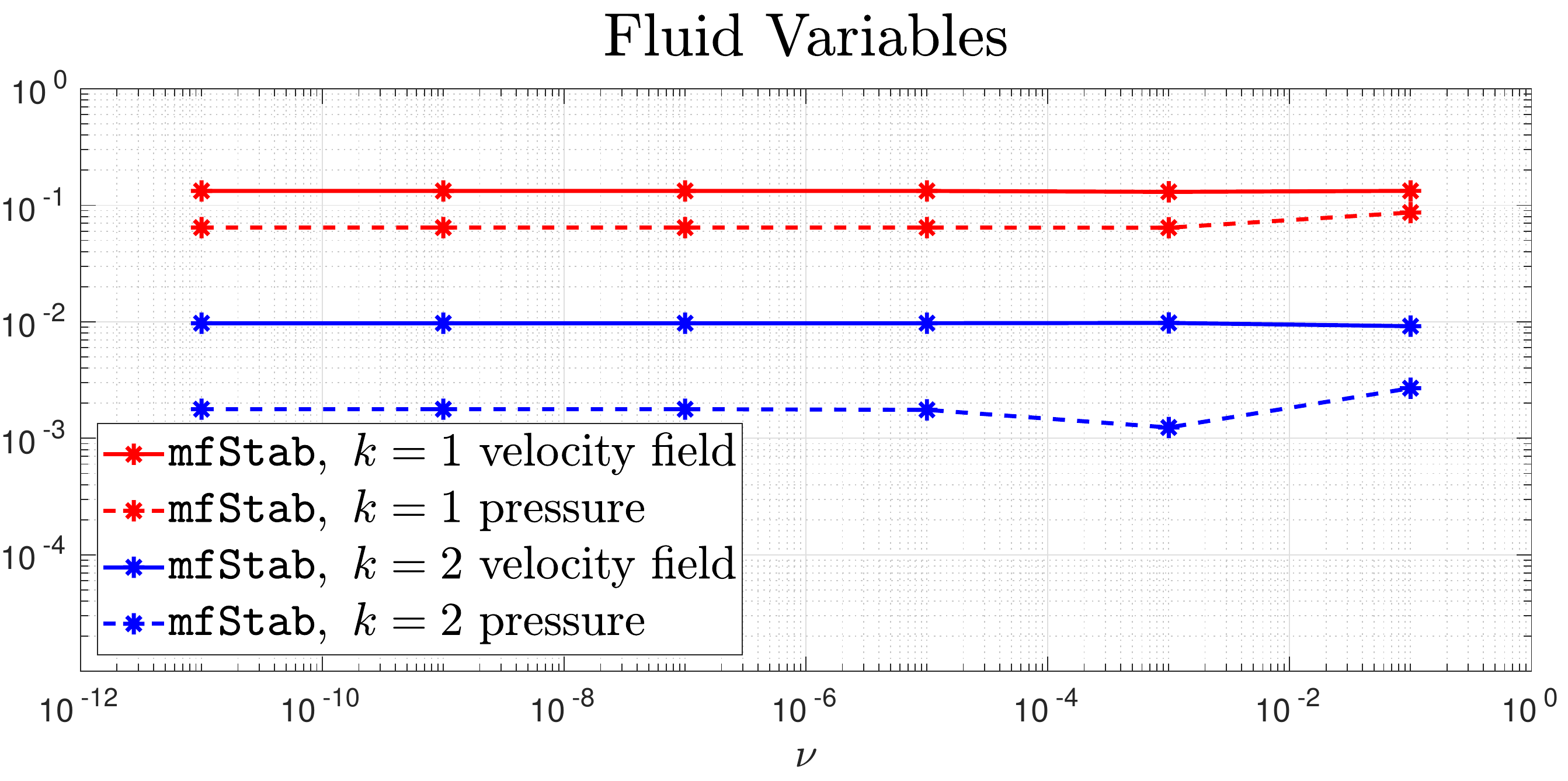}& 
\includegraphics[width=0.46\textwidth]{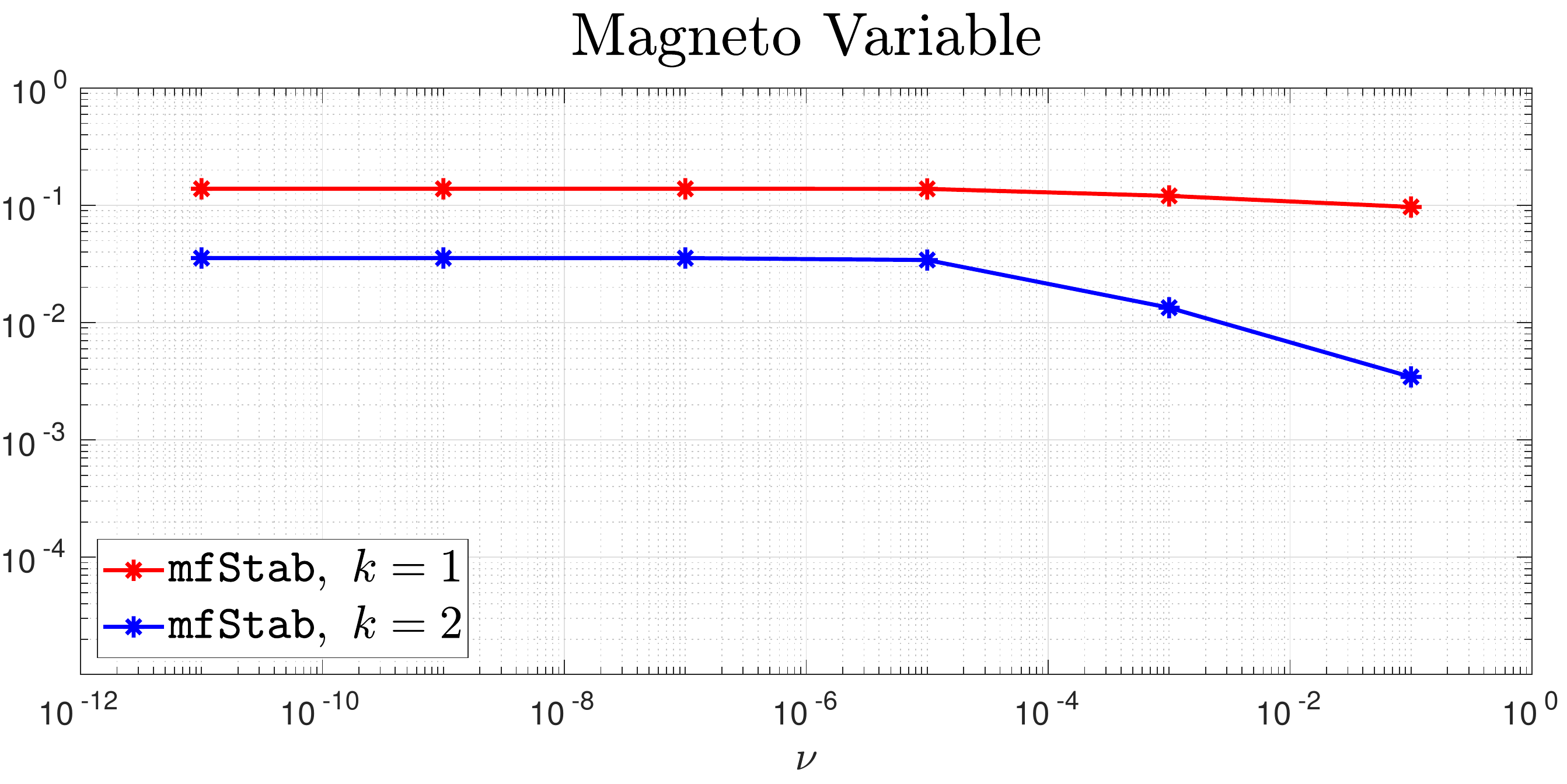}
\end{tabular}
\caption{Stability with respect to $\ns$ and $\nm$ for $k=1$ and 2.}
\label{fig:exe1_3}
\end{figure}

In Figure~\ref{fig:exe1_3}, we present an additional experiment specifically designed 
to highlight the robustness of the proposed method with respect to $\nu$.
We keep \mesh{3} fixed and we vary the values of $\nu$ form \texttt{1e-01} to \texttt{1e-11}, 
we depict only the $H^1$ semi-norm error for both velocity and magnetic fields, as well as the $L^2$ norm for pressure.
The errors remain nearly constant across different values of $\nu$, 
highlighting the fact that the \mfs{} method is not significantly affected by low values of $\nu$ (note the scale of the vertical axis in the figure).

\subsection{Magneto convective dominant regime}

In this section, we investigate a magnetic convective-dominant regime by setting $\cc=\mathbf{0}$ and $\TT=\BB$ in \eqref{eq:linear primale}.
In this scenario the stability term
$$
\bdmc \sum_{f \in \EdgesI} (\vert \cc \cdot \nn_f \vert \jump{\uu}_f , \, \jump{\vv_h}_f)_f\,,
$$
vanishes, and only the bilinear form $J_h$ stabilizes the whole method.
We set $\ns=\texttt{1e-10}$ and $\nm=\texttt{1e-02}$ 
which corresponds to a fairly realistic choice of such coefficients.

In Figures~\ref{fig:exe2_2_dgr_1} and~\ref{fig:exe2_2_dgr_2} we show the resulting convergence lines.
The method \mfs{} exhibits the optimal trends for velocity, pressure and magnetic variables, 
while the method \fs{} provides a much less accurate approximation of the velocity field.

We conducted also an experiment with $\ns=\nm=\texttt{1e-10}$, $\cc=\mathbf{0}$ and $\TT=\BB$.
Since the obtained convergence graphs are essentially identical to the ones shown in Figures~\ref{fig:exe2_2_dgr_1} and~\ref{fig:exe2_2_dgr_2}, for the sake of brevity we do not report them. 

\begin{figure}[!htb]
\centering
\begin{tabular}{cc}
\includegraphics[width=0.46\textwidth]{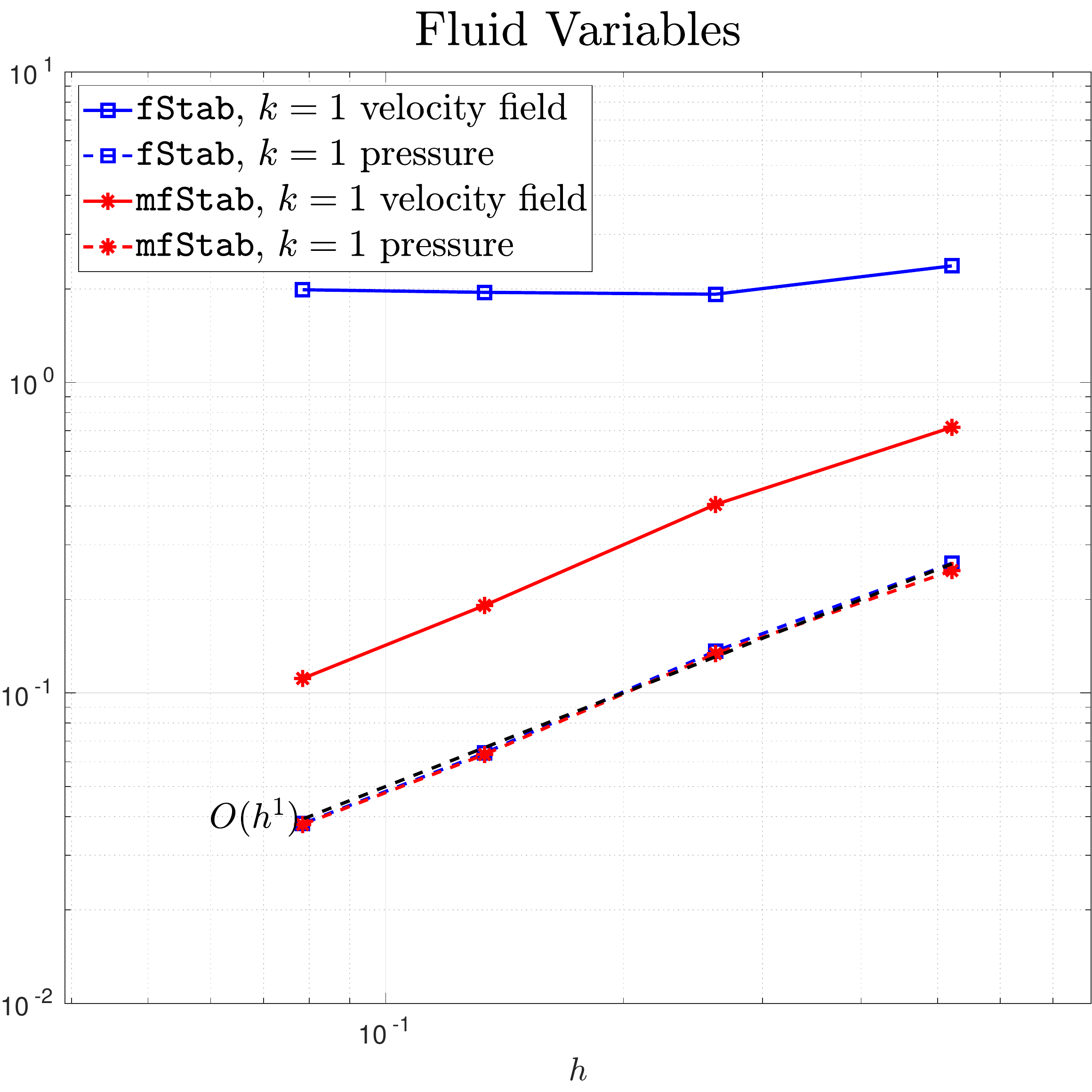}& 
\includegraphics[width=0.46\textwidth]{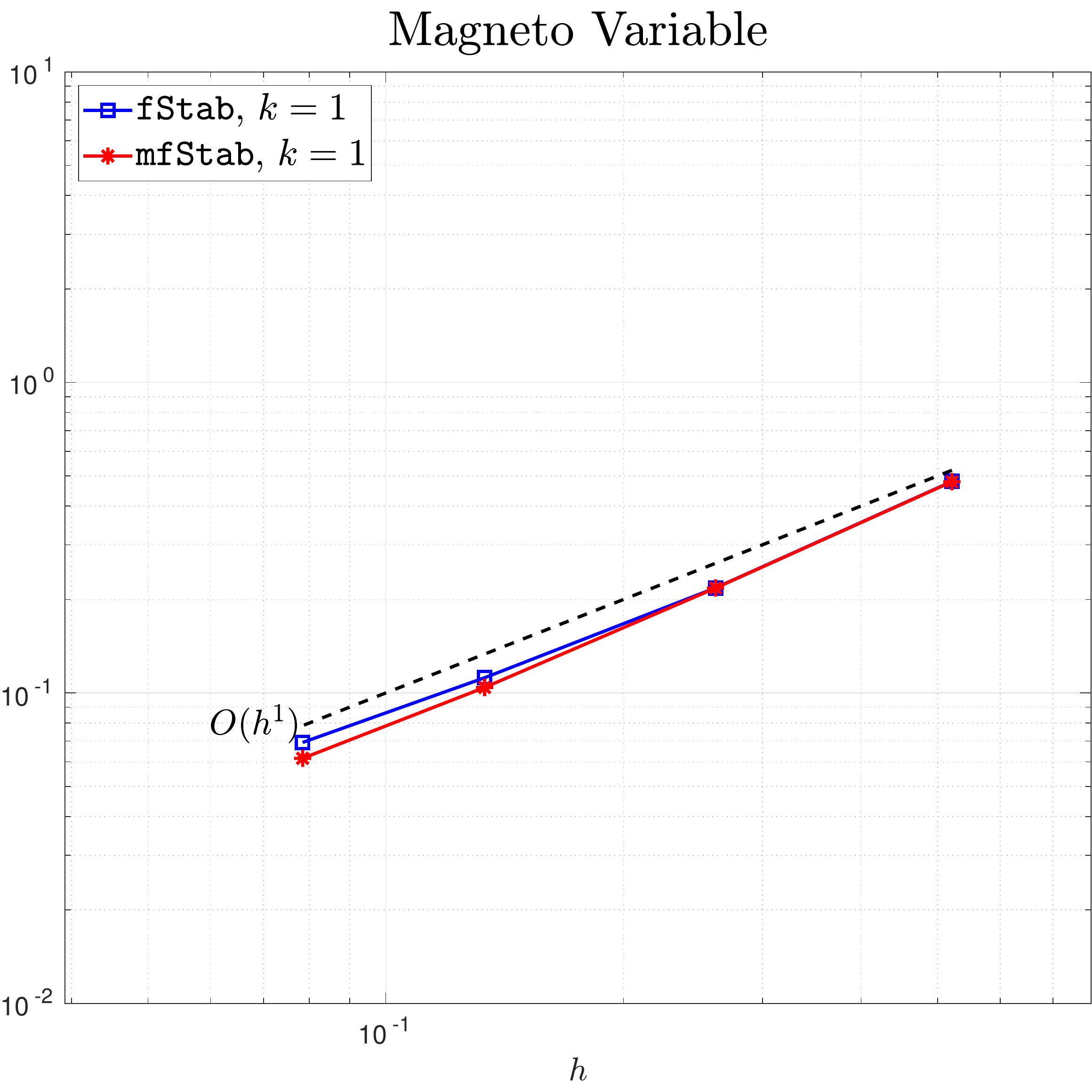}
\end{tabular}
\caption{Magneto convective dominant regime $\ns=\texttt{1e-10}$ and $\nm=\texttt{1e-02}$, $k=1$.}
\label{fig:exe2_2_dgr_1}
\end{figure}

\begin{figure}[!htb]
\centering
\begin{tabular}{cc}
\includegraphics[width=0.46\textwidth]{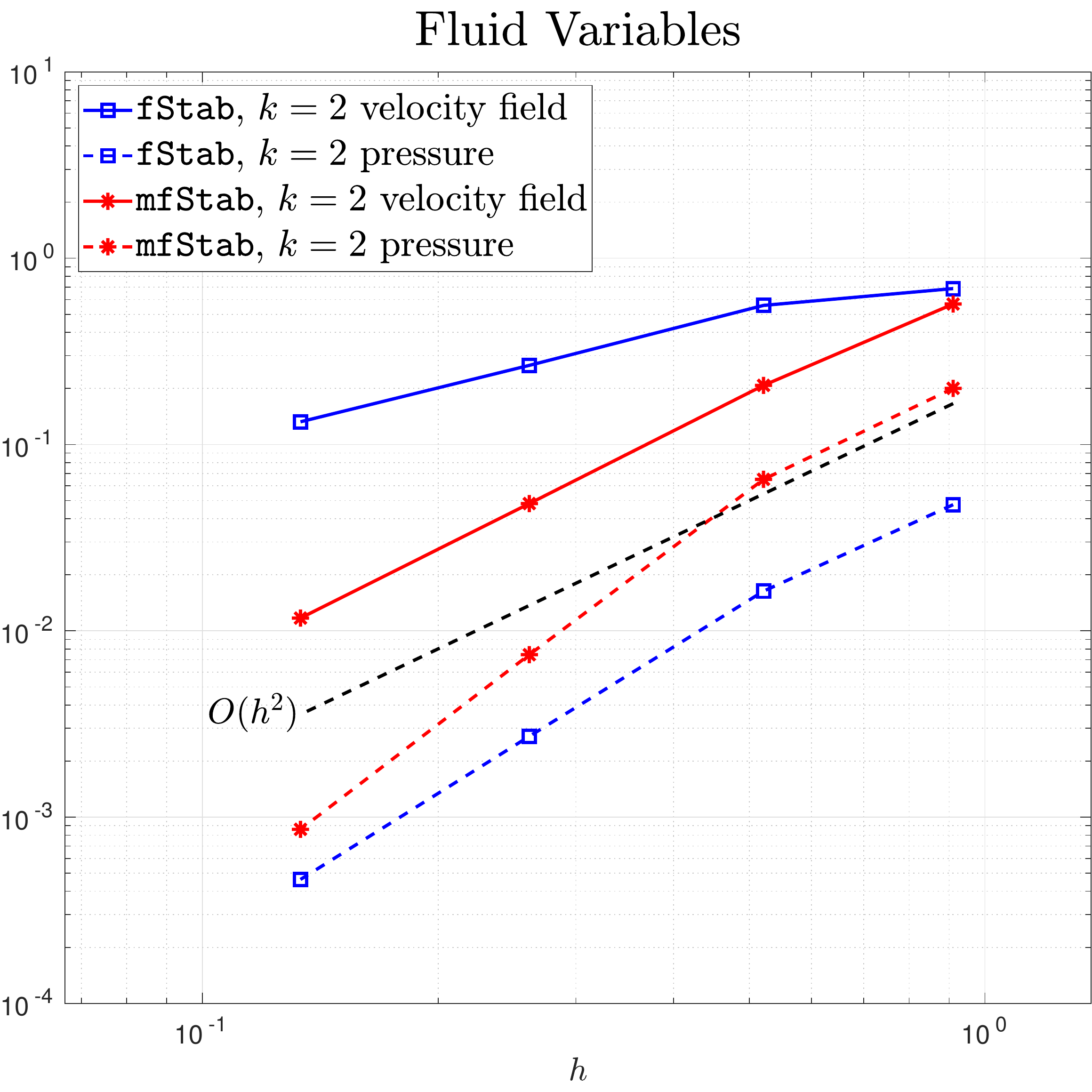}& 
\includegraphics[width=0.46\textwidth]{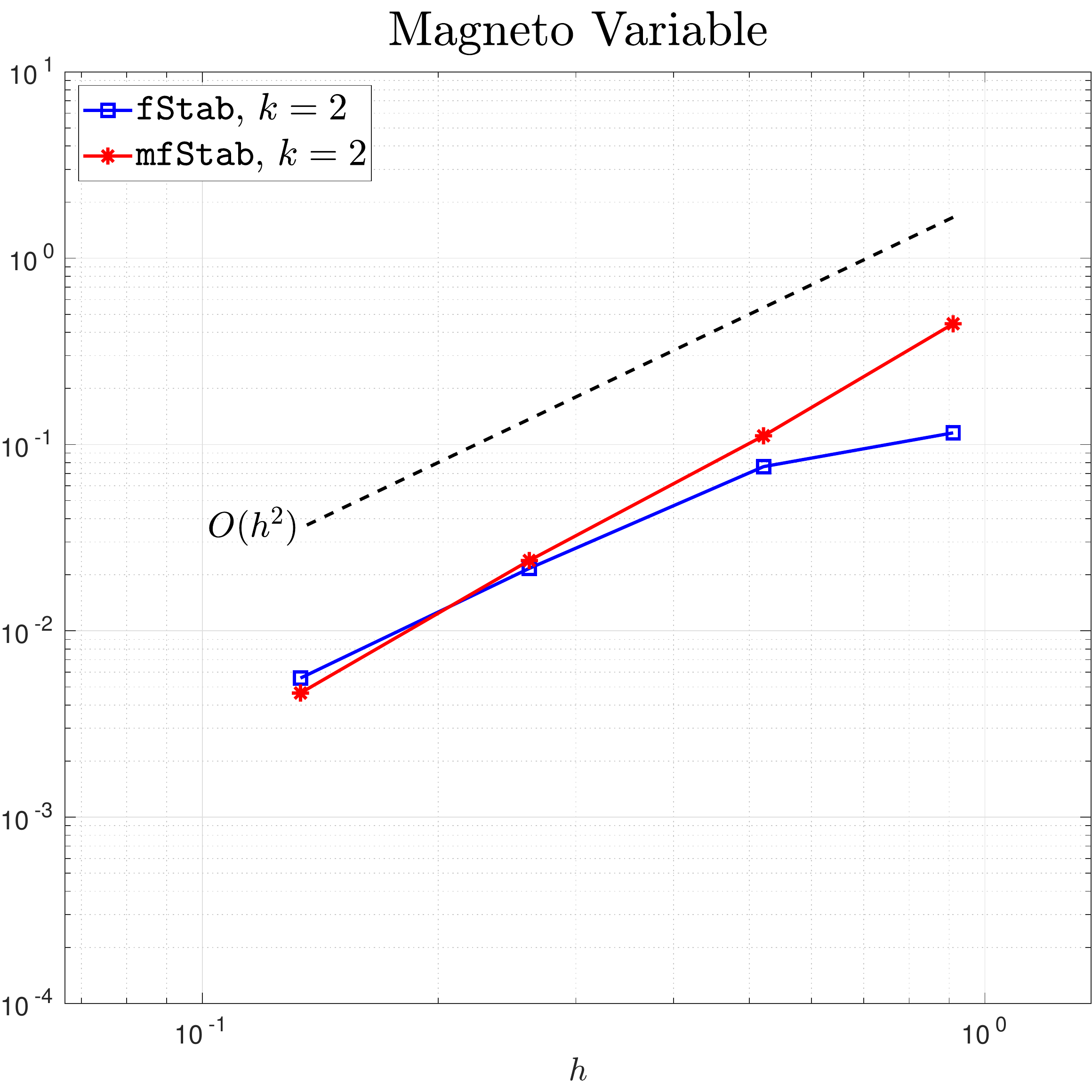}
\end{tabular}
\caption{Magneto convective dominant regime $\ns=\texttt{1e-10}$ and $\nm=\texttt{1e-02}$, $k=2$.}
\label{fig:exe2_2_dgr_2}
\end{figure}

Firstly we consider the case $\ns=\nm=\texttt{1e-10}$.
Figure~\ref{fig:exe2_1_dgr_1} illustrates the convergence lines associated with this choice of the diffusive constants.
In this convective dominant regime, the magnetic and pressure variables are unaffected, 
and both the \mfs{} and \fs{} schemes exhibit the optimal theoretical convergence rate.
However, the convergence behavior of the velocity field differs between these two schemes:
only the \mfs{} strategy has the optimal convergence rate.

\begin{figure}[!htb]
\centering
\begin{tabular}{cc}
\includegraphics[width=0.46\textwidth]{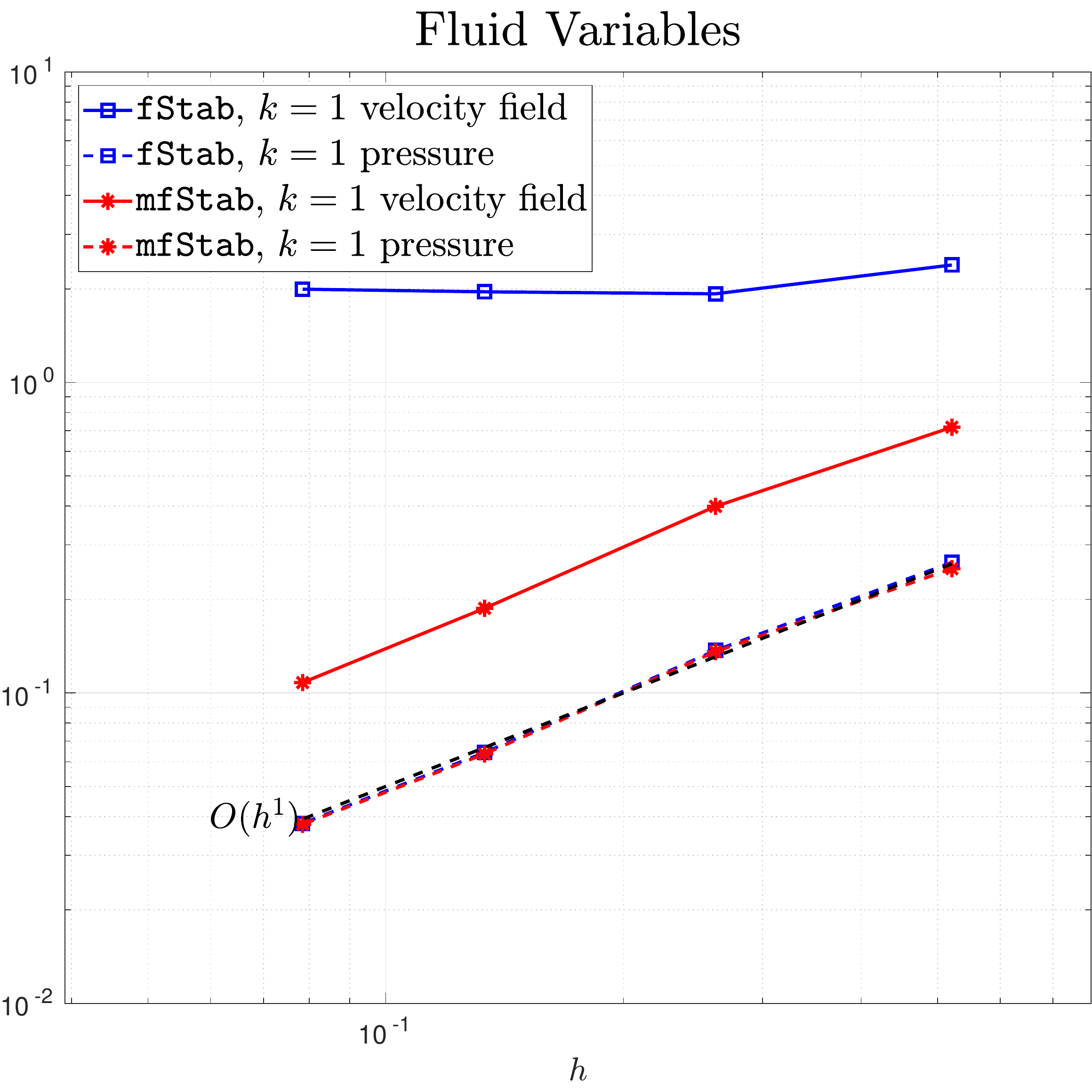}& 
\includegraphics[width=0.46\textwidth]{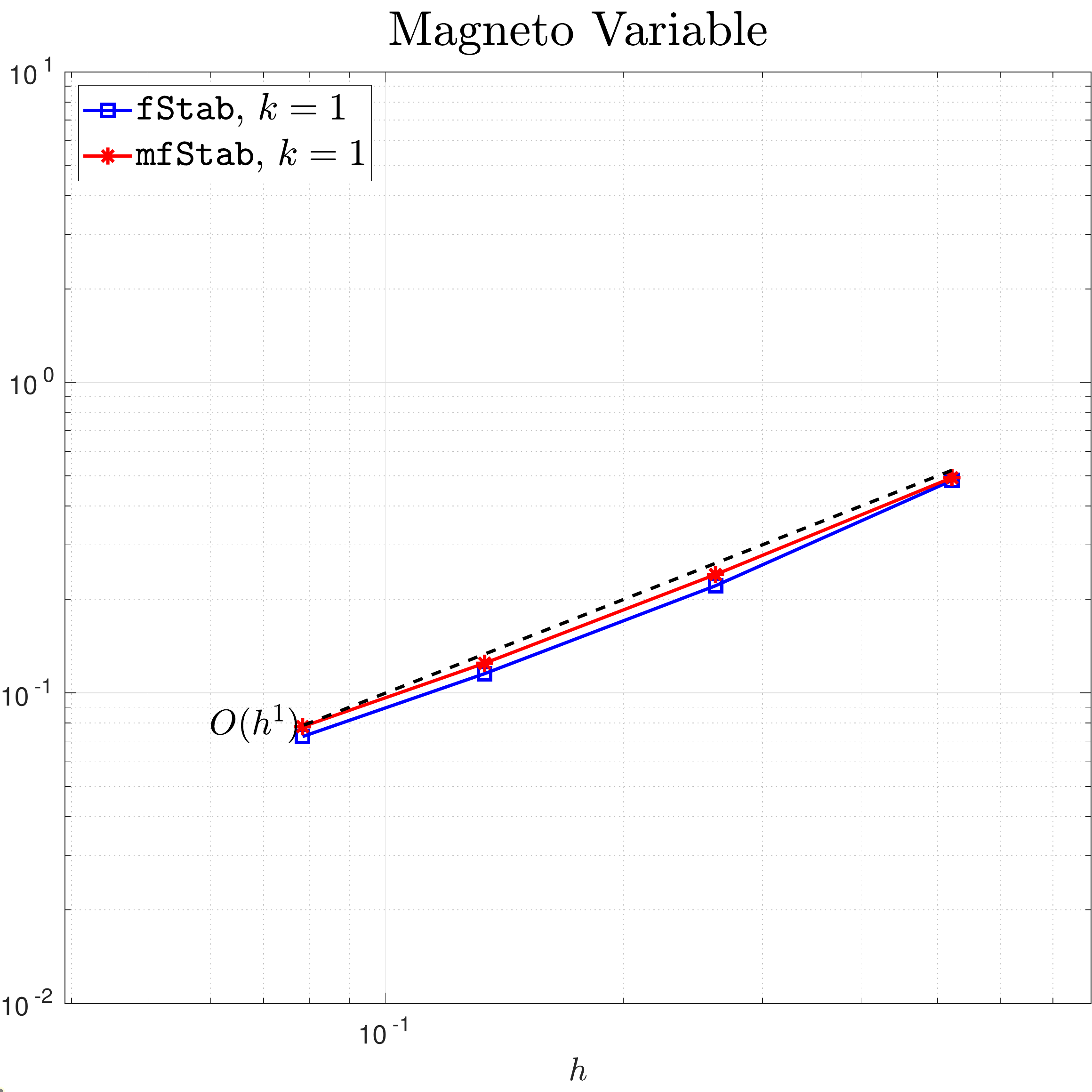}
\end{tabular}
\caption{Magneto convective dominant regime $\ns=\nm=\texttt{1e-10}$ case $k=1$.}
\label{fig:exe2_1_dgr_1}
\end{figure}

\begin{figure}[!htb]
\centering
\begin{tabular}{cc}
\includegraphics[width=0.46\textwidth]{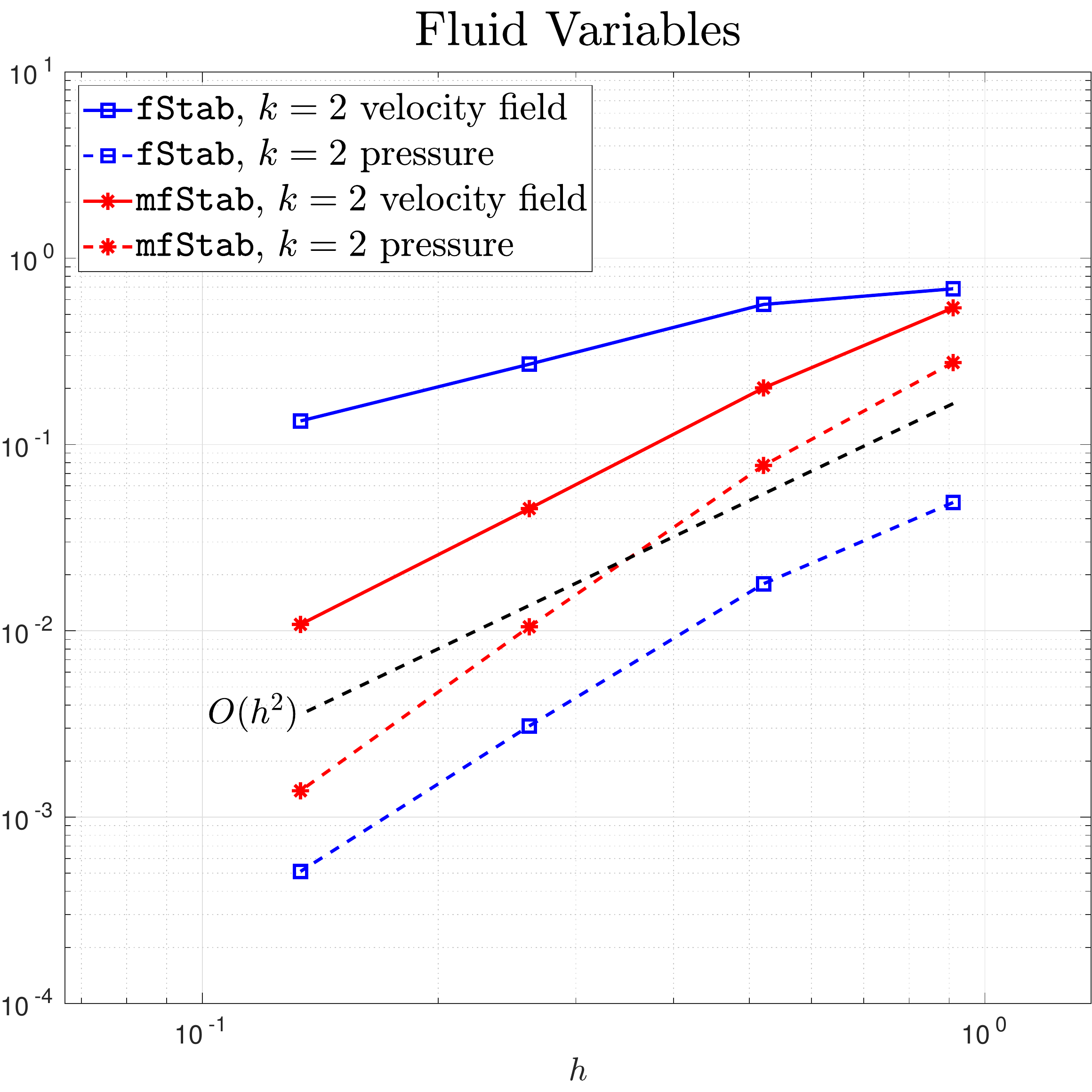}& 
\includegraphics[width=0.46\textwidth]{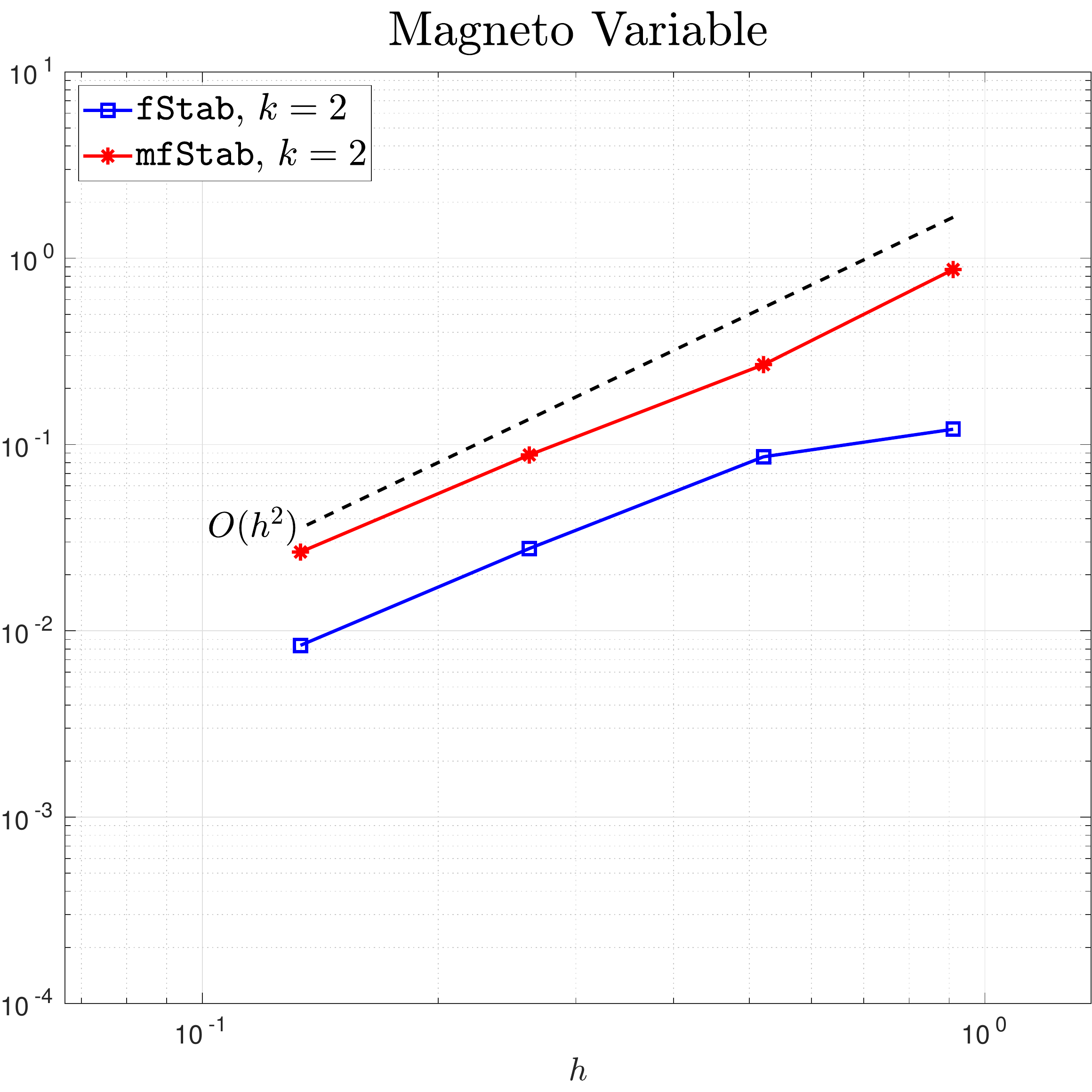}
\end{tabular}
\caption{Magneto convective dominant regime $\ns=\nm=\texttt{1e-10}$ case $k=2$.}
\label{fig:exe2_1_dgr_2}
\end{figure}

\addcontentsline{toc}{section}{\refname}
\bibliographystyle{plain}
\bibliography{references}
\end{document}